\documentclass[11pt]{article}
    \usepackage[margin=1in]{geometry}
    \usepackage[utf8]{inputenc}
    \usepackage{amsmath,amscd}
    \usepackage{mathtools}
    \usepackage{amsfonts,amsmath,amsthm,amsxtra,amssymb,mathrsfs,dsfont}
    \usepackage{framed}

    \usepackage{nth}

    \usepackage{etoolbox}
    \let\bbordermatrix\bordermatrix
    \patchcmd{\bbordermatrix}{8.75}{4.75}{}{}
    \patchcmd{\bbordermatrix}{\left(}{\left[}{}{}
    \patchcmd{\bbordermatrix}{\right)}{\right]}{}{}

    \usepackage[colorlinks = true, citecolor = blue, urlcolor = blue]{hyperref}

    \usepackage{tikz}
    \usetikzlibrary{backgrounds,calc,positioning,shapes, shadows,arrows,fit, automata}

    \usepackage{color}

    \usepackage{amsfonts}
    \usepackage{amssymb}
    \newtheorem{theorem}{Theorem}[section]
    \newtheorem*{theorem*}{Theorem}

    \theoremstyle{definition}
    \newtheorem{example}[theorem]{Example}

    \theoremstyle{remark}
    \newtheorem{remark}[theorem]{Remark}

    \newcommand{\RR}{\mathbb{R}}
    
    \newcommand{\CC}{\mathbb{C}}
    \newcommand{\ZZ}{\mathbb{Z}}
    \newcommand{\KK}{\mathbb{K}}
    \newcommand{\MM}{\mathbb{M}}
    \newcommand{\NN}{\mathbb{N}}
    \newcommand{\FF}{\mathbb{F}}

    \newcommand{\RP}{\mathbb{R}\mathbf{P}}

    \newcommand{\TT}{\mathbb{T}}
    \newcommand{\Ld}{\mathcal{L}}
    
    \newcommand{\ts}{\textsuperscript}

    \let\phi\varphi

    \DeclareMathOperator{\SW}{SW}
    
    \DeclareMathOperator{\Supp}{Supp}

\begin{document}
    \title{Twisty Takens: A Geometric Characterization of Good Observations on Dense Trajectories}
    \author{
    Boyan Xu\thanks{Department of Mathematics, University of California at Berkeley, CA, USA.} ,
    Christopher J. Tralie\thanks{Department of Mathematics, Duke University, NC, USA.} ,
    Alice Antia\thanks{Carleton College, MN, USA.} ,
    Michael Lin\thanks{Princeton University, NJ, USA.} ,
    Jose A. Perea\thanks{Department of Computational Mathematics, Science and Engineering \& Department of Mathematics Michigan State University, MI, USA.}
    }
    \maketitle

        \begin{abstract}

        In nonlinear time series analysis and dynamical systems theory, Takens' embedding theorem states that the sliding window embedding of a generic observation along trajectories in a state space, recovers the region traversed by the dynamics. This can be used, for instance, to show that sliding window embeddings of periodic signals recover topological loops, and that sliding window embeddings of quasiperiodic signals recover high-dimensional torii. However, in spite of these motivating examples, Takens' theorem does not in general prescribe how to choose such an observation function given particular dynamics in a state space. In this work, we state conditions on observation functions defined on compact Riemannian manifolds, that lead to successful reconstructions for particular dynamics. We apply our theory and construct families of time series whose sliding window embeddings trace tori, Klein bottles, spheres, and projective planes. This greatly enriches the set of examples of time series  known to concentrate on various shapes via sliding window embeddings, and will hopefully help other researchers in identifying them in naturally occurring phenomena. We also present numerical experiments showing how to recover low dimensional representations of the underlying dynamics on state space, by using the persistent cohomology of sliding window embeddings and Eilenberg-MacLane (i.e., circular and real projective) coordinates.
        \end{abstract}

\section*{Conflict of Interest}

On behalf of all authors, the corresponding author (Christopher J. Tralie) states that there is no conflict of interest. 

    \section*{Acknowledgments}

    C.J. Tralie was partially supported by an NSF big data grant DKA-1447491 and an NSF Research Training Grant NSF-DMS 1045133;
    J.A. Perea was partially supported by the NSF (DMS-1622301) and  DARPA (HR0011-16-2-003);
    all authors were supported by the NSF under Grant No. DMS-1439786 while  in residence at the Institute for Computational and Experimental Research in Mathematics in Providence, RI, during the  Summer at ICERM 2017 program on Topological Data Analysis.

    \section{Introduction}

    The \emph{delay coordinate mapping}, or \emph{sliding window embedding} \cite{takens1981detecting, nolte2010tangled, kantz2004nonlinear, das2017delay},  posits a time series as a sequence of observations made along trajectories in a hidden state space.  Under this scheme, a one dimensional time series, which could otherwise be analyzed with more traditional linear analysis techniques such as ARMA and Fourier/Wavelet analysis, is instead turned into a geometric object via a vector of samples of the time series, which moves along the signal (Equation \ref{eq:slidingwindow}). The shape of this geometric object provides information about the system under study.  Periodic processes, for example, map to points which concentrate on a topological loop.  Sliding window embeddings have been used in this context, for example, to analyze ECG signals of a beating heart \cite{stam2005nonlinear, plesnik2014detection}, to detect chatter in mechanical systems \cite{khasawneh2016chatter}, to quantify repetitive motions in human activities \cite{frank2010activity,venkataraman2016persistent}, to discover periodicity in gene expression during circadian rhythms \cite{perea2015sw1pers}, and to detect wheezing in audio signals \cite{emrani2014real}.  In addition to loops, torus shapes often show up during ``quasiperiodicity,'' which is a state of near-chaos.  Sliding window embeddings have witnessed this torus shape in such applications as vocal fold anomalies \cite{herzel1994analysis}, horse whinnies \cite{briefer2015segregation}, neural networks \cite{morrison2016diversity}, and oscillating cylinder flow \cite{glaz2017quasi}.  Certain time series even concentrate on fractals after a sliding window embedding \cite{takens1981detecting, de2012topological}.  Sliding window embeddings have also been used as a tool for shape analysis more generally even when an underlying model for the dynamics is unknown, such as in music structure analysis \cite{bello2011measuring, serra2009cross}.
    We direct the interested reader to 
    \cite{perea2019topological} for a recent review on how topological data analysis can be used in the analysis of time delay embeddings.

    The main theory motivating the use of sliding window embeddings in all of these applications is Takens' delay embedding \cite{takens1981detecting} theorem, which is stated as follows:

    \begin{theorem*}[Takens' embedding theorem \cite{takens1981detecting}]\label{thm:takens}
        Let $M$ be a compact manifold of dimension $m$. Suppose $X$ is a smooth vector field with flow $\psi_t: M \rightarrow M$ and $G$ is a smooth function on $M$. For $\tau>0$, $N \geq 2m$, and pairs $(X,G)$ it is a generic property that $\Psi^N_\tau: M \rightarrow \RR^{N+1}$ defined by
        \[ \Psi^N_\tau(p) = (G(p), G(\psi_\tau(p)), G(\psi_{2\tau}(p)), \ldots, G(\psi_{N\tau}(p)))\]
        is an embedding.
    \end{theorem*}
    A ``random'' choice of $X$ and $G$ makes the delay coordinate mapping $\Psi^N_\tau$ a smooth embedding. Thus, remarkably, the state space $M$ of a dynamical system may in general be reconstructed from a single generic observation function $G$\footnote{Some texts refer to this as an ``observable.''}, which gives rise to a 1D time series. However, in practice, Takens' result is ill-suited for computational purposes because it does not provide an explicit characterization of ``genericity''.
    In this work, we extend Takens' embedding theory with a geometric characterization of observations which yield high-dimensional delay coordinate embeddings, given a particular flow on a manifold. Our main theoretical result for general compact manifolds is stated in Theorem \ref{thm:main} in Section~\ref{sec:goodobservations}, as follows:
    
     \begin{theorem*}
    The Takens map $\Psi_{\tau}^{N}$ is an embedding for some dimension $N>0$ and flow time $\tau>0$, if the following conditions hold:
        \begin{enumerate}
        \item For any point of $p\in M$ there is an $m$-tuple $J\in \ZZ_{\geq 0}^m$ of nonnegative integers such that the $m$-form
        \[\Ld_{X}^{\wedge J}dG := \bigwedge_{j\in J}\Ld_X^jdG\]
        is nonzero at some point on the integral curve $\gamma_p(s)$. Here, $\mathcal{L}_X^j$ denotes the $j^{\textnormal{th}}$-order Lie derivative.

        \item For any pair of distinct points $p, q\in M$ the observation curves $g_p(s)$ and $g_q(s)$ are not identical.
        \end{enumerate}
    \end{theorem*}

    We first provide several examples in Section~\ref{sec:distancesobs} which satisfy the conditions of our theorem.  In the process, we discuss a non-example that violates condition 1 if we're not careful (Example~\ref{ex:spheredist}) and show another non-example which violates condition 2 (Example~\ref{ex:distklein}, part 2).  We then prove our theorem in Section~\ref{sec:goodobservations}, and we explore a special case in Section~\ref{sec:fourier} in which Fourier bases can be used to construct observation functions\footnote{The code to generate all figures in this manuscript can be found at \url{http://www.github.com/ctralie/TwistyTakens}}.

    \section{Background}


    In this section, we provide a more detailed overview of several  concepts utilized in this work, including sliding window embeddings, persistent (co)homology, and Eilenberg-MacLane coordinates.  The latter two tools will be used to empirically validate that our sliding window embeddings recover our chosen state space
     and the underlying dynamics.


    \subsection{Sliding Window Embeddings}

    We express a time series $g(t)$ as an observation $G$ along a dense trajectory $\gamma$ on a manifold $M$, i.e.
    \[g(t) = G(\gamma(t))\]
    for $\gamma : \RR \longrightarrow M$ and $G: M \longrightarrow \RR$.
    We compute the \emph{sliding window} of $g$ as

    \begin{equation}
    \label{eq:slidingwindow}
    \SW^N_\tau g(t) := \left[ \begin{array}{c} g(t) \\ g(t + \tau) \\ g(t + 2 \tau) \\ \vdots \\ g(t + N\tau) \end{array} \right] \in \mathbb{R}^{N+1}
    \end{equation}
    where $N\in \NN$ is the number of delays, $\tau > 0$ is the delay time, and $N \tau$ is the window length.


    We interpret the sliding window $\SW^N_\tau g(t)$ as the evaluation of the Takens map $\Psi^N_\tau$ in Theorem \ref{thm:takens} above on an integral curve $\psi_t(p)$ of a vector field $X$ through a point $p\in M$.
    For if $N=2\cdot\dim{M}$ and $\gamma(t) = \psi_t(p)$, then
    \[\SW_\tau^Ng(t) = \Psi^N_\tau(\psi_t(p)).\]
    For sufficiently large $N$ and small $\tau$, $\SW_\tau^Ng(t)$  densely ``traces'' the embedding $\Psi^N_\tau(M)$ for appropriate choice of observation $G$
    and vector field $X$.


    When $g$ is a periodic function with frequency $\omega\in \RR$, it readily follows
    that the sliding window embedding $\SW^N_\tau g(t)$  traces a closed curve in $\RR^{N+1}$.
    The shape of this curve is closely related to the choice of parameters $N$ and $\tau$, and their relation to $\omega$ \cite{perea2015sliding}.
    In particular, if $\tau$ and $N$ are chosen so that $N$ is large enough and
    $N\tau\omega \approx 1$, then the image of $\SW^N_\tau g$ is in fact a topological circle in $\RR^{N+1}$, whose shape
    is tightly controlled by the Fourier coefficients of $g$.
    In other words, the periodic nature of $g$ ---  a spectral property --- is reflected in the circularity of its sliding window,
    a topological feature.
    Quasiperiodicity is another spectral notion with  a  clear geometric/topological counterpart.
    Indeed, let
     $1,\omega_1, \ldots, \omega_n\in \RR$ be linearly independent over the rational numbers.
    We say that $f: \RR \longrightarrow \RR$ is quasiperiodic with frequencies $\omega_1,\ldots, \omega_n$, if
    it can be written as
    $f(t) = F(t, \ldots, t)$ for some function $F: \RR^n \longrightarrow \RR$  whose  $j$-th marginals $f_j(t) = F(t_1 ,\ldots, t_{j-1}, t , t_{j+1},\ldots, t_n)$ are periodic with frequency $\omega_j$.
    In this case, and for appropriate $N$ and $\tau$, the set $\SW_\tau^N f(\ZZ)$ is dense in an $n$-dimensional torus
    embedded  in $\RR^{N+1}$ \cite{perea2016persistent, hitesh2018}.

    \subsection{Koopman spectra}

    We now review another relevant tool that goes along with sliding window embeddings. For positive flow time $t>0$, the flow $\psi_t$ of a vector field $X$ on a compact manifold $M$ defines a diffeomorphism $\psi_t: M \rightarrow M$. Then the composition map $U^t$, or Koopman operator \cite{Koopman315, das2017delay, Mezic} given by
    \[U^tG = G \circ \psi_t,\]
    is a linear operator on the space of observation functions on $M$. The coordinates of the delay mapping are thus iterated applications of $U^t$ on an observation $G$.

    For certain classes of dynamical systems, the Koopman operator possesses a discrete spectrum and yields a linear expansion
    \[G = \sum_{k=0}^\infty G_k\phi_k\]
    where $\phi_k$ are eigenfunctions of $U^t$ and $G_k$ are \textit{Koopman modes}. For such systems one ``lifts'' the dynamics on the state space to an evolution of observables. For a more comprehensive overview of Koopman theory and its applications, please refer to \cite{koopmanbackground}.

    We will see in Section \ref{sec:fourier} that a high-dimensional delay mapping essentially recovers the Koopman modes of an observation function. We therefore characterize delay embedding observations in terms of spectral decomposition properties. We examine a special case with a Fourier basis for the Koopman operator on the Torus and Klein bottle, and show via our main Theorem \ref{thm:main}  what is needed of these coefficients.

    \subsection{Persistent Homology}

    In practice we evaluate the sliding window $\SW^N_\tau g(t)$ at a finite set of evenly sampled time points $t_1 < \cdots < t_J$.
    This results in a discrete collection of $J$ vectors, referred to as a ``sliding window point cloud''.
    The topology of a point cloud with $J$ points is trivial; it consists of $J$ connected components
    and lacks any other topological features (loops, voids, etc).
    However, if we use a {\em simplicial complex} (a discrete object) to  approximate the underlying space from which the point cloud is sampled,
    then we can estimate the underlying topology via combinatorial means.
    A simplicial complex on a set $V$ of vertices (e.g., a sliding window point cloud) is a collection  $K$ of nonempty subsets
    $\sigma \subset V$, so that if $\emptyset \neq \tau \subset \sigma \in K$, then $\tau \in K$.
    As an example, suppose we seek a simplicial complex with  topology reflecting that of the unit circle $S^1$.
    Starting with the set $V = \{a, b, c\}$ of vertices, we let
    $K = \{a, b, c, \{a, b\}, \{b, c\}, \{a, c\} \}$ be the simplicial complex containing 3 edges between every pair of vertices.
    Like $S^1$, $K$ has one connected component,
    one loop which bounds an empty space, and no higher dimensional features (voids, etc).

    So far, our description of simplicial complexes has been purely combinatorial/topological,
    but one can use geometry to inform their construction.
    An early scheme in Euclidean space is  the alpha complex \cite{edelsbrunner1994three},
    constructed as a family of subcomplexes of  Delaunay triangulations at different scales.
    An even simpler construction, which works  in any metric space, is the so-called ``Vietoris-Rips''  complex at scale $\alpha \geq 0$,
    denoted $R_\alpha(V)$.
    It is comprised of the finite subsets of $V$ which have diameter less than $\alpha$.
    Choosing the ``appropriate'' scale  is ill-posed.
    For instance, Figure \ref{fig:TDAExample} shows a point cloud in $\mathbb{R}^2$ for which it is impossible to choose an appropriate scale at which the simplicial complex contains the two empty loops that are present in the original shape.

    \begin{figure}[htb!]
        \centering
        \includegraphics[width=\textwidth]{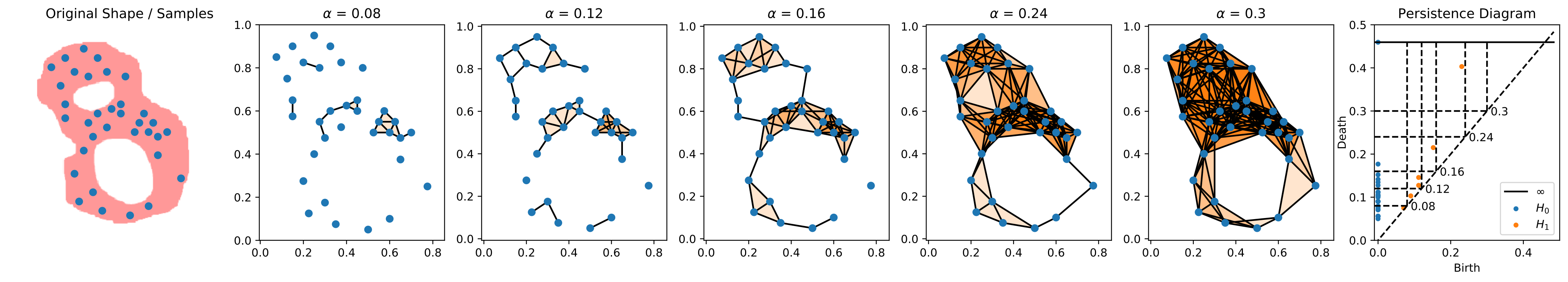}
        \caption{An example of the Rips filtration on a point cloud sampled from a thickened figure eight.
         The Rips complex is shown at different scales, the values of $\alpha$, producing the persistence diagram on the right.}
        \label{fig:TDAExample}
    \end{figure}

    Specifically, $R_{0.16}(V)$ contains the upper loop, but not the lower loop, and $R_{0.24}(V)$ contains the lower loop, but the upper loop is no longer empty.
    In fact, it is impossible to  choose an $\alpha$ in which both loops are present and empty in $R_{\alpha}(V)$ in this example.
    However, we can still summarize the multiscale topological information of any point cloud by performing a {\em filtration} of the complex.
    That is, we evaluate $R_{\alpha}(V)$ as $\alpha$ varies continuously from $0$ to some maximum value,
    so that $R_{\alpha_1}(V) \subset R_{\alpha_2}(V)$ if $\alpha_1 \leq \alpha_2$.
    Throughout this process we keep track of topological features as they appear, or are ``born,'' and as they are filled in, or ``die''.
    For each such {\em homology class}, we can produce a point in a scatter plot, known as the {\em persistence diagram} of the filtration,
    with birth time on the $x$-axis and death time on the $y$-axis.
    Figure \ref{fig:TDAExample} shows a persistence diagram associated with our running example\footnote{We compute persistence diagrams for all examples in this paper using the Python interface to ``Ripser'' \cite{ripser, ripserpy}}.  Intuitively, points further from the diagonal correspond to larger topological features which ``persist'' (stay alive) over longer intervals, and points closer to the diagram correspond to small, ``noisy'' features which are often artifacts of sampling (e.g. the square and pentagon loop that exist at $\alpha = 0.12$).

    For completeness, we extend the above explanation with a brief rigorous presentation.  For a more comprehensive treatment, please refer to  \cite{edelsbrunner2008persistent, edelsbrunner2010computational, carlsson2009topology, ghrist2014elementary, perea2018brief}.
    Let $(\Gamma, \preceq )$ be a partially ordered set.
    A $\Gamma$-filtered simplicial complex is a collection $\mathcal{K} = \{K_\alpha\}_{\alpha \in \Gamma}$ of simplicial complexes,
    so that $K_\alpha \subset K_{\alpha'}$ for every $\alpha \preceq \alpha' \in  \Gamma$.
    The typical examples for point cloud data are the Rips filtration, as we mentioned,
    and the \v{C}ech  filtration motivated by the nerve lemma \cite{Hatcher}.
    Specifically, let $L$ be  a finite subset of a metric space $(\MM,\mathbf{d})$.
    The Rips filtration of $L$ is the  $\RR$-filtered simplicial complex
     $\mathcal{R}(L) = \{R_\alpha(L)\}_{\alpha \in \RR}$.
    Similarly,  for $\ell \in L$ let
    \[
    B_\alpha(\ell) = \{b\in \MM : \mathbf{d}(b,\ell) < \alpha\}
    \;\;\;\;\; \mbox{ and } \;\;\;\;\;
    \mathcal{B}_\alpha = \{B_\alpha(\ell) : \ell \in L\}.\]
    The \v{C}ech complex
    $\check{C}_\alpha(L)$ is defined as the nerve
    of $\mathcal{B}_\alpha$; that is
    $\check{C}_\alpha(L) = \mathcal{N}(\mathcal{B}_\alpha)$
    where
    \[
    \sigma \in \mathcal{N}(\mathcal{B}_\alpha)
    \hspace{.5cm}
    \mbox{ if and only if }
    \hspace{.5cm}
    \bigcap\limits_{\ell \in \sigma} B_\alpha(\ell) \neq \emptyset
    \]
    Hence $\check{C}(L) = \{\check{C}_\alpha\}_{\alpha \in \RR}$
    is an $\RR$-filtered simplicial complex, and
    $R_\alpha(L) \subset \check{C}_\alpha(L) \subset R_{2\alpha}(L)$
    for all $\alpha \in \RR$.

    The persistent homology (resp. cohomology) of a filtered complex $\mathcal{K} = \{K_\alpha\}_{\alpha \in \Gamma}$,
    with coefficients in a field $\FF$, are  defined, respectively, as
    \[
    PH_n(\mathcal{K}; \FF) :=
    \bigoplus\limits_{\alpha \in \Gamma} H_n(K_\alpha;\FF)
    \hspace{1cm} \mbox{and} \hspace{1cm}
    PH^n(\mathcal{K}; \FF) :=
    \bigoplus\limits_{\alpha \in \Gamma} H^n(K_\alpha;\FF)
    \]
    Let $\iota_{\alpha,\alpha'}: H_n(K_\alpha ; \FF) \longrightarrow H_n(K_{\alpha'};\FF)$
    and $\jmath^{\alpha',\alpha}: H^n(K_{\alpha'} ; \FF) \longrightarrow H^n(K_{\alpha};\FF)$
    be the $\FF$-linear maps induced by the inclusion $K_\alpha \subset K_{\alpha'}$, $\alpha \preceq \alpha' $.
    A persistent homology  (resp. cohomology) class is an element
    $\bigoplus_{\alpha \in \Gamma} \nu_\alpha \in PH_n(\mathcal{K}; \FF)$
    (resp. $\bigoplus_{\alpha \in \Gamma} \mu^\alpha \in PH^n(\mathcal{K}; \FF)$ )
    so that $\iota_{\alpha, \alpha'} (\nu_\alpha) = \nu_{\alpha'}$
    (resp $\jmath^{\alpha', \alpha} \left(\mu^{\alpha'}\right) = \mu^{\alpha}$)
    for every $\alpha \preceq \alpha'$.

    When $\Gamma = \RR$,  a theorem of Crawley-Boevey \cite{crawley2015decomposition} contends  that if each
    $H_n(K_\alpha; \FF)$ is finite-dimensional (also known in the literature as being pointwise-finite)
    then one can choose bases $S^\alpha$ for each $H_n(K_\alpha ; \FF)$, satisfying the following compatibility condition:
    \begin{enumerate}
    \item $\iota_{\alpha,\alpha'}(S^\alpha) \subset \left(S^{\alpha'} \cup \{0\}\right)$ for every $\alpha \leq \alpha'$.
    \item If $\iota_{\alpha,\alpha'}(\mathbf{v}_j^\alpha) = \iota_{\alpha,\alpha'}(\mathbf{v}_k^\alpha)$ and $j\neq k$, then
    $\iota_{\alpha,\alpha'}(\mathbf{v}_j^\alpha) = 0$.
    \end{enumerate}
    The set $S = \bigcup\limits_{\alpha \in \RR}S^\alpha$ admits a partial order $\preceq$ given by
    $S^\alpha \ni \mathbf{v} \preceq \mathbf{v}' \in  S^{\alpha'}$
    if and only if $\alpha \leq \alpha'$ and $\iota_{\alpha,\alpha'}(\mathbf{v}) = \mathbf{v}'$.
    The maximal chains in $(S,\preceq)$ are the persistent homology classes.
    To each maximal chain $\mathscr{C} \subset S$ one can associated the point $(b_\mathscr{C}, d_\mathscr{C}) \in [-\infty, \infty] \times [-\infty, \infty]$
      defined by
    \[
    b_{\mathscr{C}} = \inf\{\alpha \in \RR : S^\alpha \cap \mathscr{C}\neq \emptyset\}
    \;\;\;\; ,
    \;\;\;\;
    d_{\mathscr{C}} = \sup\{\alpha \in \RR : S^\alpha \cap \mathscr{C}\neq \emptyset\}
    \]
    The collection of such pairs, where $\mathscr{C}$ runs over all maximal chains, is the persistence diagram for the persistence homology of the filtered
    complex $\mathcal{K}$.

    Persistent cohomology behaves similarly. Indeed, any basis for $H_n(K_\alpha ; \FF)$ yields a well-defined isomorphism
    $H_n(K_\alpha; \FF) \cong H_n(K_\alpha;\FF)^*$ with the linear dual space,
    and the latter is naturally isomorphic to $H^n(K_\alpha ; \FF)$, by the universal coefficient theorem.
    Hence, these isomorphisms turn the $S^\alpha$'s into a collection of compatible bases for the
    cohomology groups $H^n(K_\alpha; \FF)$, showing that persistent homology and cohomology yield the same persistence diagrams.

    \subsubsection{Persistent Homology of Sliding Window Embeddings}
    As mentioned in the introduction, there are numerous examples in the literature of persistent homology on sliding window point clouds.
    For any periodic time series ($x(t) = x(t + kT), k \in \mathbb{Z}$),
    a sliding window embedding yields a topological loop, and there is a point of high persistence in the persistence diagram for $PH_1$ \cite{perea2015sliding}.
    However, the authors of \cite{perea2015sliding} also show, surprisingly, that sliding window embeddings of functions like
    $x(t) = \cos(t) + a \cos(2t)$, $|a| > 1$, can lie on the boundary of an embedded M{\"o}bius strip \cite{perea2015sliding}.
    We use this to help intuitively explain the time series we obtain for the projective plane (Example~\ref{ex:projplane}) and the Klein Bottle
    (Section~\ref{sec:kleinbottle}).
    Note that this also means that field coefficients other than $\mathbb{Z}_2$ are needed to maximize the maximum persistence in $PH_1$.
    In general, for $\cos(t) + a \cos(kt)$, coefficients which are not prime factors of $k$ are needed \cite{perea2015sliding, traliemoebius}.
    Finally, there are works which utilize both $PH_1$ and $PH_2$ to quantify the presence of quasiperiodicity in time series data, by estimating 
    the toroidality of a sliding window point cloud \cite{perea2016persistent, tralie2017quasi}.
    In this work, we extend this suite of examples beyond (possibly twisted) loops and torii to  other manifolds.

    \subsection{Eilenberg-MacLane Coordinates}

    Though persistent homology is informative, one can further utilize it to perform nonlinear dimensionality reduction on  sliding window point clouds, for  visualization purposes and reconstruction of the underlying dynamics.  To this end, we use ``Eilenberg-MacLane coordinates'',  which turn persistent cohomology classes into maps from point clouds  to the circle \cite{de2011persistent,  perea2018circular}, and   (real or complex) projective spaces \cite{ perea2018multiscale}.  
    We present next a more detailed  summary;
    maps to the projective plane are particularly interesting, as they allow us to ``untwist'' non-orientable manifolds like the Klein bottle.

    More formally, if $G$ is  an abelian group \footnote{In this section $G$ will refer to an Abelian group, but it otherwise refers to an observation function.} and $n$  is a positive integer, then it is possible to construct a connected CW complex $K(G,n)$, called an Eilenberg-MacLane space,  whose homotopy type is uniquely determined by two properties:

    \begin{enumerate}
    \item
     its $j$-th homotopy group $\pi_j(K(G,n))$ is trivial for all $j\neq n$
    \item $\pi_n(K(G,n))\cong G$
    \end{enumerate}

    The Brown representability theorem (for CW complexes and singular cohomology) contends that if $B$ is a CW complex, then there is a natural bijection
    \begin{equation}\label{eq:BrownRep}
    H^n(B;G) \cong [B, K(G,n)]
    \end{equation}
    between the $n$-th cohomology of $B$ with coefficients in $G$,
    and the set of homotopy classes of maps from $B$ to $K(G,n)$.

    The two Eilenberg-MacLane spaces we use to generate circular and projective coordinates are:
    $K(\ZZ,1) \simeq S^1$, and  $K(\ZZ/2, 1) \simeq \RP^\infty = \RR^\infty \smallsetminus\{\mathbf{0}\}/\sim$, respectively.
    Here
    $\RR^\infty$ is the collection of infinite sequences of real numbers
    $\mathbf{x} = (x_0, x_1 , \ldots)$ which are nonzero for all but finitely many $x_j$'s, and
    $\mathbf{x} \sim \mathbf{y}$ if and only if $\mathbf{x} = r\mathbf{y}$
    for some $r\in \RR\smallsetminus\{0\}$.
    One can also regard $\RR^\infty$  as  the direct limit of the system  $\RR \subset \RR^2 \subset \RR^3 \subset \cdots$,
    where the inclusion $\RR^j \hookrightarrow \RR^{j+1}$ sends
    $(x_0,\ldots, x_{j-1})$ to $(x_0,\ldots, x_{j-1},0)$.
    With this interpretation in mind,  $\RP^\infty$ can be regarded as the direct limit of the system $\RP^0 \subset \RP^1 \subset \RP^2 \subset \cdots $,
    where $\RP^n = \RR^{n+1}\smallsetminus \{\mathbf{0}\}/\sim$.
    Recently \cite{perea2018multiscale}, it has been shown that if $L$ is a finite subset of a metric space
    $(\MM,\mathbf{d})$,
    and for $\ell \in L$ we let $B_\alpha (\ell)$ be the open ball of radius $\alpha$ centered at $\ell$,
    then persistent cohomology classes in $PH^1(\mathcal{R}(L);\ZZ/2)$ can be used to define projective coordinates
    \[
    f_\mu : \bigcup\limits_{\ell \in L} B_{\alpha}(\ell) \longrightarrow \RP^n
    \]
    Similarly,  persistent cohomology classes in $PH^1(\mathcal{R}(L); \ZZ/q)$, for appropriate choices of prime $q>2$, yield circular coordinates \cite{perea2018circular}
    \[
    f_{\theta, \tau}: \bigcup\limits_{\ell \in L} B_{\alpha}(\ell ) \longrightarrow S^1
    \]
    In both cases, the resulting coordinates mimic the properties of the bijection (\ref{eq:BrownRep}) from Brown's representability.

    \subsubsection{Projective Coordinates}
    \label{sec:projcoords}
    Here is a sketch of the construction of projective coordinates from persistent cohomology classes.
    Let $L = \{\ell_0, \ldots, \ell_n\} \subset \MM$,   and fix a cocycle
     $\mu=\{\mu^\alpha_{jk} \} \in Z^1(R_{2\alpha}(L); \ZZ/2)$
    so that its cohomology
    class is not in the kernel of the homomorphism
    \[
    \iota^{2\alpha, \alpha}:
    H^1(R_{2\alpha}(L);\ZZ/2) \longrightarrow H^1(R_\alpha(L);\ZZ/2)
    \]
    induced by the inclusion $R_\alpha(L) \subset R_{2\alpha}(L)$.
    Since $R_\alpha(L) \subset \check{C}_\alpha(L) \subset R_{2\alpha}(L)$,
    then the rightmost inclusion yields a nonzero class in
    $H^1(\check{C}_\alpha(L);\ZZ/2)$.
    We let
    \[
    \begin{array}{cccl}
      f_\mu : & \bigcup\limits_{\ell \in L} B_\alpha(\ell) = L^{(\alpha)} & \longrightarrow  & \RP^n \\
       & B_{\alpha}(\ell_j) \ni b & \mapsto &
       \left[
       (-1)^{\mu_{j0}^\alpha}|\alpha - \mathbf{d}(b, \ell_0)|_{+}
       : \cdots :
       (-1)^{\mu_{jn}^\alpha}|\alpha - \mathbf{d}(b, \ell_n)|_{+}
       \right]
    \end{array}
    \]
    where $[x_0 : \cdots : x_n] \in \RP^n$ denotes the equivalence class of $(x_0,\ldots, x_n)\in \RR^{n+1}\smallsetminus \{\mathbf{0}\}$,
    and $|r|_{+}:= \max\{0,r\}$ for $r\in \RR$.
    Since $\{\mu_{jk}^\alpha\}$ is a cocycle,
    it readily follows that the point
    $f_\mu(b)\in \RP^n$ is independent of the index
    $j\in \{0,\ldots, n\}$ for which $b \in B_{\alpha}(\ell_j)$.
    In other words, $f_\mu$ is well defined.

    If $\{\nu^\alpha_{jk}\}\in Z^1(R_{2\alpha}(L) ;\ZZ/2)$
    is cohomologous to $\{\mu_{jk}^\alpha\}$,
    and $f_\nu : L^{(\alpha)} \longrightarrow \RP^n$ is the associated map,
    then
    $f_\mu \simeq f_\nu$
    and hence we get a well defined
    function
    \[
    \begin{array}{ccl}
    H^1(R_{2\alpha}(L) ; \ZZ/2 ) & \longrightarrow &\left[L^{(\alpha)}, \RP^n\right]  \\
    \left[\mu\right]&  \mapsto & [f_\mu]
    \end{array}
    \]
    The metric properties of  $f_\mu$ are also
     determined by the cohomology
    class of $\mu$.
    For if  
    \[
    \mathbf{d}_g(\mathbf{x}, \mathbf{y}) :=
    \arccos
    \left(\frac{
    |\langle \mathbf{x}, \mathbf{y} \rangle|
    }
    {\|\mathbf{x}\|\cdot \|\mathbf{y}\|}\right)
    \]
    denotes the geodesic distance in $\RP^n$,
     then it readily follows that
    \[
    \mathbf{d}_g\left(f_\nu(b), f_\nu(b')\right)
    =
    \mathbf{d}_g\left(f_\mu(b), f_\mu(b')\right)
    \]
    for all $b,b'\in L^{(\alpha)}$ and $\mu, \nu$ in the same cohomology class.

    Given a finite set $P \subset L^{(\alpha)}$, taking its image
    through $f_\mu$ yields a new point cloud
    $f_\mu(P) \subset \RP^n$.
    A dimensionality-reduction scheme in
    $\RP^n$ referred to as  principal projective component
    analysis is also defined in \cite{perea2018multiscale}.
    This procedure yields
    a sequence of maps
    \[
    P_k : f_\mu(P) \longrightarrow \RP^k, \hspace{1cm}
     k = 0,  \ldots, n
     \]
    minimizing an appropriate notion of (metric) distortion.
    In particular, $P_k\circ f_\mu(P)$ and
    $P_k \circ f_\nu(P)$ are isometric if $\mu$ and $\nu$ are cohomologous.
    The point clouds
    $P_k \circ f_\mu(P) \subset \RP^k$
    are referred to as the projective
    coordinates of $P$, induced by the landmarks
    $L \subset \MM$ and the cohomology class
    $\left[\mu \right] \in H^1(R_{2\alpha}(L); \ZZ/2)$.

    As an example, Figure~\ref{fig:ProjectiveCoordsExample} shows the projective coordinates onto $\RP^2$ of points sampled from a Klein bottle $\KK$, using the flat metric on the torus $\TT$, descended onto the automorphism $\kappa: (x,y) \mapsto (x+\pi, -y)$.
    We use the cocycle representative which is the sum of the  representative cocycles from the two most persistent classes.

    \begin{figure}[!htb]
        \centering
        \includegraphics[width=0.9\textwidth]{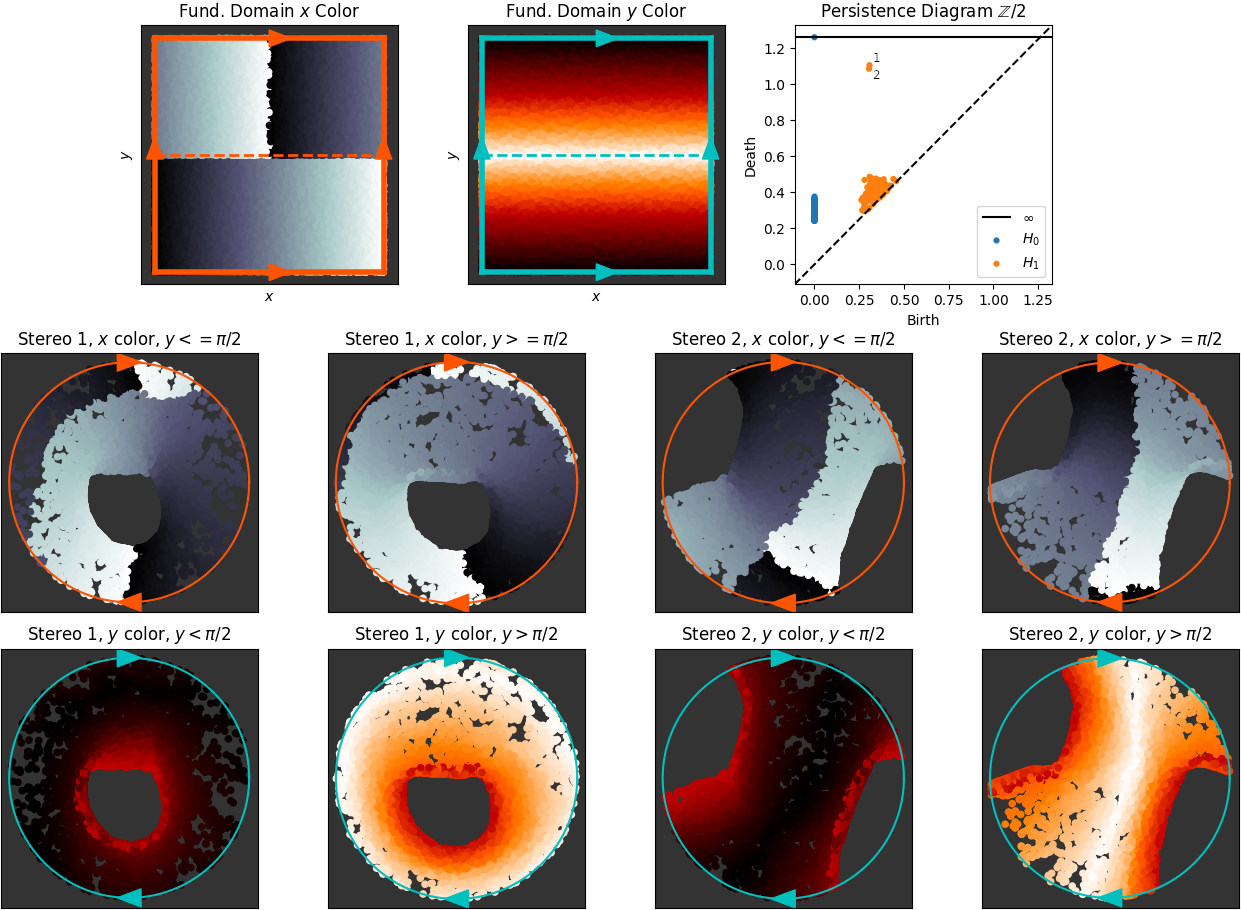}
        \caption{An example of projective coordinates for point sampled from a flat Klein bottle   obtained as a quotient of the torus via $[x, y] \sim [x + \pi, -y]$.  The two coordinates are colored according to their $x$ and $y$ positions on the fundamental domain $[0, 2 \pi] \times [0, \pi]$, and we show two different stereographic projections to the plane from $\RP^2$.  If the fundamental domain is split into distinct parts $A = [0, 2 \pi] \times [0, \pi/2]$ and $B = [0, 2 \pi] \times [\pi/2, \pi]$, then $A$ and $B$ map to two distinct M{\"o}bius strips which are attached at their boundaries at $y = \pi/2$ (medium red for the y colors), which is indeed what happens when the Klein bottle is cut down the middle.}
        \label{fig:ProjectiveCoordsExample}
    \end{figure}

    In fact, this is a 2 to 1 map, as shown in Figure~\ref{fig:ProjectiveCoordsExample}.  Just as a torus can be obtained from gluing two annuli together at their boundary, the Klein bottle can be obtained by gluing two M{\"o}bius strips at their boundary.  Each one of these M{\"o}bius strips is visible in the odd and even columns of the bottom two rows of Figure~\ref{fig:ProjectiveCoordsExample}, respectively.  In particular, the loops $[0, 2 \pi] \times 0$ and $[0, 2 \pi] \times \pi$ are at the center of each M{\"o}bius strip, and the boundaries of each M{\"o}bius strip at $[0, 2 \pi] \times \pi/2$ get identified at the center of the projective coordinates plot.  We will observe similar projective coordinates for the sliding window of our Klein bottle time series in Section~\ref{sec:kleinbottle}.

    \subsubsection{Circular Coordinates}

    The idea of using the bijection
    $H^1(B;\ZZ) \cong [B , S^1]$ to construct circle-valued functions
    for data, from persistent cohomology classes,
    was first introduced by de Silva et. al. \cite{de2011persistent}.
    Their construction has shortcomings (not sparse, not transductive)
    which are addressed in \cite{perea2018circular};
    the latter is the procedure we use in the paper
    and the one we describe next.

    Let $q>2$ be a prime so that the homomorphism
    \[H^1(R_{2\alpha}(L) ;\ZZ) \longrightarrow H^1(R_{2\alpha}(L);\ZZ/q)\]
    induced by the projection $\ZZ \longrightarrow \ZZ/q$,
    is surjective.
    Hence, any
    $\mu\in Z^1(R_{2\alpha}(L);\ZZ/q)$
     has a lift 
    $\tilde{\mu}\in Z^1(R_{2\alpha}(L);\ZZ)$.
     Moreover, if $\iota : \ZZ \hookrightarrow \RR $
    is the inclusion homomorphism, then
    there are cochains
    $\theta \in Z^1(R_{2\alpha}(L);  \RR)$
    and $\tau \in C^0(R_{2\alpha}(L) ; \RR)$
    so that
    $\theta$ is the unique harmonic cocycle
    representative of $\iota^*([\tilde{\mu}])$ and
    $\iota^\#(\tilde{\mu}) = \theta - \delta^0\tau$. From this data we define
    \[
    \begin{array}{cccl}
      f_{\theta,\tau} : & \bigcup\limits_{\ell \in L} B_\alpha(\ell)  & \longrightarrow  & S^1\subset \CC \\
       & B_{\alpha}(\ell_j) \ni b & \mapsto &
       \exp
       \left\{
       2\pi i
       \left(
       \tau_j + \sum\limits_{k=0}^n \theta_{jk}\varphi_k (b)
       \right)
       \right\}
    \end{array}
    \]
    where
    \[
    \varphi_k (b)= \frac{|\alpha - \mathbf{d}(b,\ell_k)|_{+}}
    {\sum\limits_{r=0}^n |\alpha - \mathbf{d}(b,\ell_r)|_{+}}
    \]

    Figure~\ref{fig:CircularCoords} shows an example of this algorithm on a point cloud sampled from a torus, using 400 landmarks.  In this example, the algorithm is able to find maps from the points to the inner and outer circle of the torus.

    \begin{figure}[!htb]
        \centering
        \includegraphics[width=\textwidth]{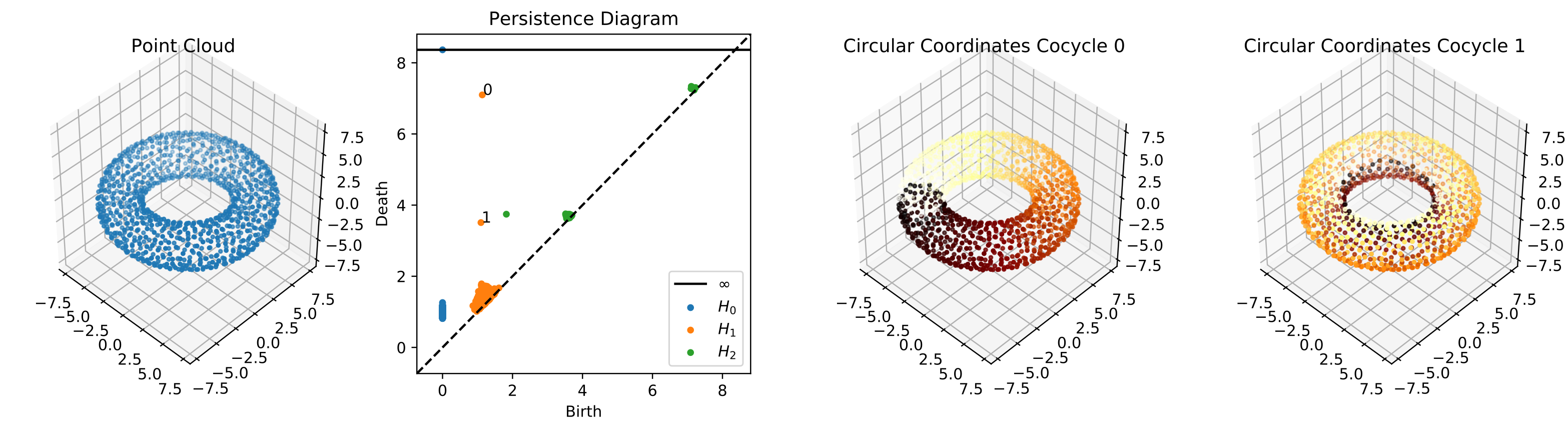}
        \caption{An example of the circular coordinates algorithm on a point cloud sampled from a torus in $\mathbb{R}^3$.  The third plot shows the coordinates resulting from the representative cocycle of the largest persistence class, which goes around the large circle on the outside, while the fourth plot shows the circular coordinates resulting from the cocycle from the second largest persistence class, which wraps around the inner circle.}
        \label{fig:CircularCoords}
    \end{figure}

    \section{Preliminary Examples: Distance To A Point As Observation Function}
    \label{sec:distancesobs}
    To motivate a more general development of good observation functions on manifolds, we first explore a very specific genre of observation functions: 
    those which arise as the distance to a specified point in the manifold.  We then verify the geometric integrity of a delay coordinate mapping of the resulting time series using persistent homology and Eilenberg-Maclane coordinates on a few examples.  Through these tools and a visual comparison of the time series to known examples, we will already be able to explain quite a lot, including motivating both conditions of Theorem~\ref{thm:main}, though a full development of the theory in Section~\ref{sec:goodobservations} is needed to justify these choices of observation functions.
    
    In the discussion below, all of our observation functions are of the form $G(x) = d(x, \hat{x})$, where $d$ is some metric chosen on the manifold and $\hat{x}$ is some fixed point on the manifold which is our ``reference distance point.''

    \begin{example}
    Flat torus $\TT$

    \begin{figure}[!htb]
        \centering
        \includegraphics[width=\textwidth]{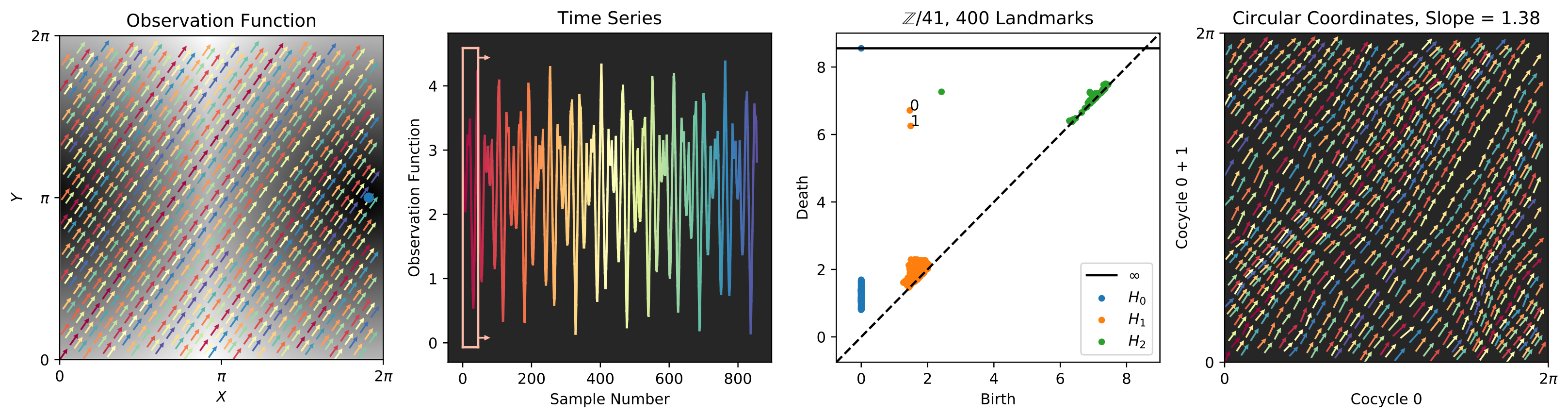}
        \caption{An irrational winding on the flat torus, with an observation function as the distance to the point $\hat{x} = (6, \pi)$, which is shown as a blue dot on the left plot.  The distance from this point is indicated in gray (dark means close, light means far).  The resulting time series is shown in the second plot, with a sliding window indicated with a red box.  The third and fourth figures show, respectively, the persistence diagrams of the sliding window point cloud and the resulting circular coordinates.
        The arrows in the fourth plot are the recovered dynamics; they indicate the order on the sliding windows inherited from the time series.  Colors are coordinated between the flows in the first, second, and fourth plots.  Similar plotting conventions are present in Figures ~\ref{fig:KleinDist},~\ref{fig:KleinDistFail},~\ref{fig:SphereDist},~\ref{fig:ProjDist},~\ref{fig:TwoHoledTorusDist}.}
        \label{fig:TorusDist}
    \end{figure}

    We first examine the planar torus $\TT=\RR^2/2\pi\ZZ^2$, parameterized by $(u, v) \in [0, 2 \pi] \times [0, 2 \pi]$.  As our dynamics, we take the irrational winding  $\psi_t(u, v) = (u + \sqrt{2} dt, v + dt)$, and the observation $G(u, v)$ is the flat geodesic distance between $(u, v)$ and the point $\hat{x} = (6, \pi)$.  This is shown in Figure~\ref{fig:TorusDist}.  After performing a delay embedding on the resulting time series with window length of 30 samples, we see two persistent $H_1$ classes and 1 persistent $H_2$ class, which is the signature of a torus.  Furthermore, circular coordinates resulting from the top two persistent classes in $H_1$ recovered the full original flow specification.
    \end{example}

    \begin{example}
    \label{ex:distklein}

    Flat Klein Bottle $\KK$

    \begin{figure}[!htb]
        \centering
        \includegraphics[width=\textwidth]{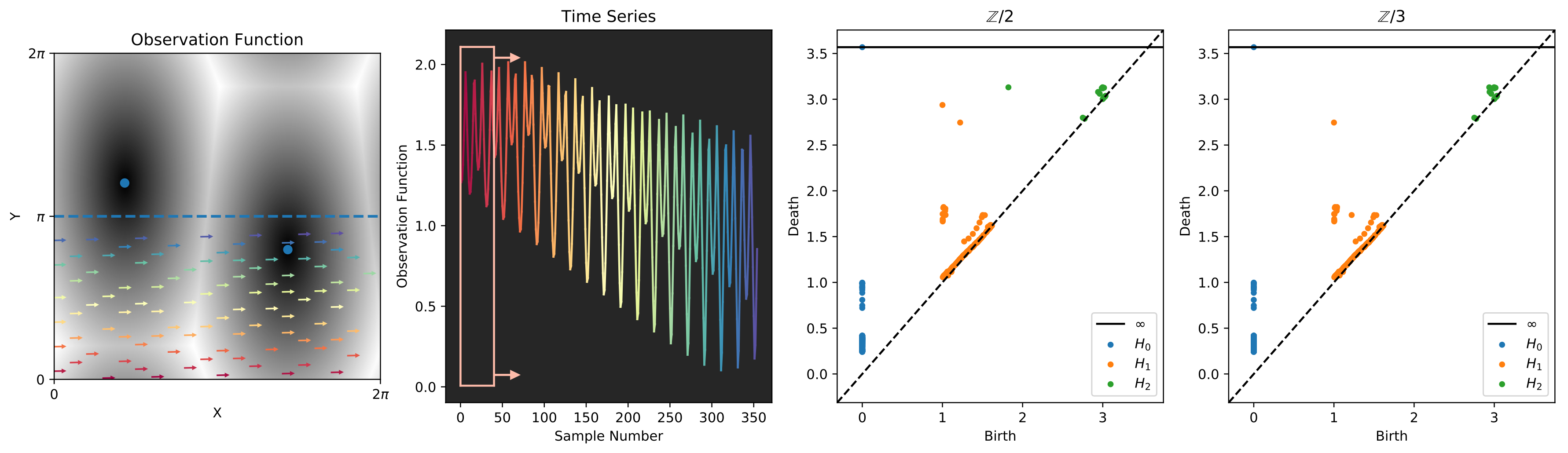}
        \caption{A winding with a very shallow slope on the fundamental domain of a flat Klein bottle, which is double covered by the flat torus by the automorphism $(x, y) \sim (x + \pi, -y)$.  The observation function is then a scaled $L^2$ distance from the point $\hat{x} = (4.5, 2.5)$, which descends under the automorphism.}
        \label{fig:KleinDist}
    \end{figure}

    As in our projective coordinates example in Figure~\ref{fig:ProjectiveCoordsExample}, we now form a quotient on the domain of the flat torus to create a Klein bottle, via the automorphism $\kappa: (x, y) \sim (x + \pi, -y)$.  Then, the metric on the torus descends to the Klein bottle via $\kappa$.  We use a slightly modified weighted $L^2$ flat metric as our distance measure for the observation function; that is

    \begin{equation}
    d_{\alpha, \beta}((u_1, u_1), (u_2, u_2)) = \sqrt{ \alpha^2 (u_1 - u_2)^2 + \beta^2 (u_1-u_2)^2}
    \end{equation}

    In this particular example, we let $\alpha = 1$ and $\beta = 0.5$, and we take an observation to the point $\hat{x} = (4.5, 2.5)$; that is, $G(u, v) = d_{1, 0.5}((u, v), (4.5, 2.5)$.  Finally, we use a flow with a very shallow slope, $\psi_t(u, v) = (u + dt, v + 0.05dt)$, in the fundamental domain $y<\pi$. After performing a sliding window embedding with a window length of 30 samples, we see two persistent classes in $H_1$ and one persistent class in $H_2$ with $\mathbb{Z}/2$ coefficients, but we only see one class in $H_1$ and no classes in $H_2$ with $\mathbb{Z}/3$ coefficients.  This is indeed the signature of a Klein bottle.  We will show projective coordinates on a similar example with a slightly different observation function in Section~\ref{sec:kleinbottle}, and we will explain more intuitively visual features of the time series at that point.

    \begin{figure}[!htb]
        \centering
        \includegraphics[width=\textwidth]{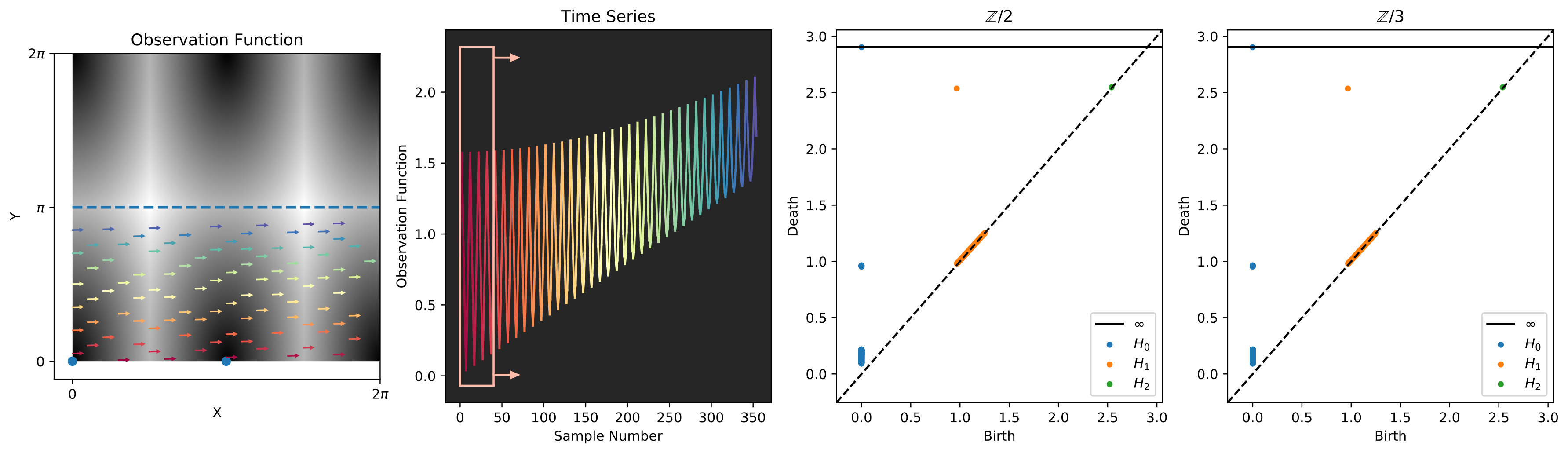}
        \caption{Not all distance functions on the Klein bottle work.  The conditions here are the same as in Figure~\ref{fig:KleinDist}, but the point $\hat{x}$ from which distance is measured has been moved to $(\pi, 0)$.  The sliding window embedding degenerates to a cylinder in this example.}
        \label{fig:KleinDistFail}
    \end{figure}

    Note that not every distance function will lead to a reconstruction of the Klein bottle.  For instance, if we use the same flow $\psi_t$ but an observation function $G(u, v) = d((u, v), (\pi, 0))$, as in Figure~\ref{fig:KleinDistFail}, then the sliding window embedding of the resulting time series degenerates to a cylinder, because there exist pairs of points with the same observation curves under the flow.  This motivates condition 2 in Theorem~\ref{thm:main}.
    \end{example}

    \begin{example}
    \label{ex:sphere}
    Sphere $\mathcal{S}^2$

    \label{ex:spheredist}
    \begin{figure}[!htb]
        \centering
        \includegraphics[width=0.65\textwidth]{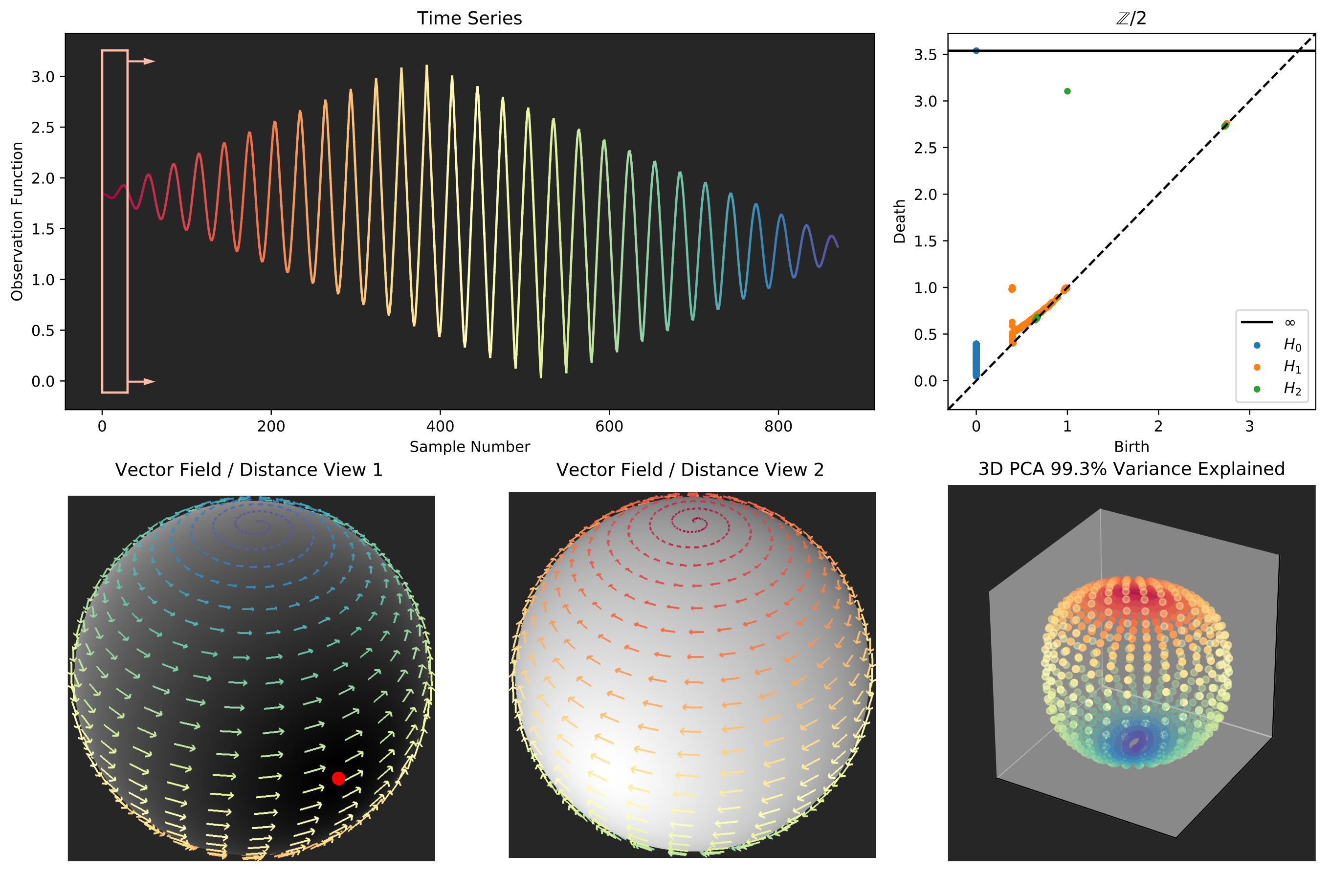}
        \caption{An observation function on the sphere which is the geodesic distance from a point $\hat{x}$ drawn in red.  The top and bottom views of the vector field are drawn in the left two figures.  3D PCA of the sliding window embedding, which retains nearly all of the variance of the sliding window point cloud, is shown in the bottom right plot.}
        \label{fig:SphereDist}
    \end{figure}

    We now reconstruct the sphere from a given trajectory and distance function.  Tralie \cite{tralie2017Dissertation} showed empirically that a sliding window embedding of a helical trajectory, under the observation function on the sphere which is the arclength from some point on the sphere, yields an embedding of the sphere.  We replicate this here.  More specifically, we parameterize the unit sphere in spherical coordinates $(\phi, \theta)$ (where $\phi$ is azimuth and $\theta$ is elevation from the north pole), we let $\psi_t(\phi, \theta)_{\alpha} = (\phi + dt, -\pi/2 + \theta dt)$, and the let the observation $G(\phi, \theta)$ to a point $\hat{x} = (\hat{\phi}, \hat{\theta})$ be
    
    \begin{equation}
        G(\theta, \phi) = \cos^{-1} \left( \cos(\phi)\sin(\theta)\cos(\hat{\phi})\sin(\hat{\theta}) + \sin(\phi)\sin(\theta)\sin(\hat{\phi})\sin(\hat{\theta}) + \cos(\theta)\cos(\hat{\theta}) \right)
    \end{equation}
    
    We repeat this here in Figure~\ref{fig:SphereDist}.  In this example, simple linear dimension reduction via PCA is able to recover the most of the geometry of the sliding window point cloud, though spherical coordinates are also possible in the Eilenberg-MacLane framework \cite{perea2018multiscale}.

    One pitfall in this example is that the observation point $\hat{x}$ cannot lie on the equator or the north or south poles; that is, $\hat{\phi} \notin \{ -\pi/2, 0, \pi/2 \}$. In these cases, the helix structure is flattened to a spiral,so the sliding window embedding degenerates to a disc.  This motivates the ``derivative rank'' condition, or condition 1 in Theorem~\ref{thm:main}.

    \end{example}

    \begin{example}
    \label{ex:projplane}
    Projective plane $\RP^2$

    \begin{figure}[!htb]
        \centering
        \includegraphics[width=0.65\textwidth]{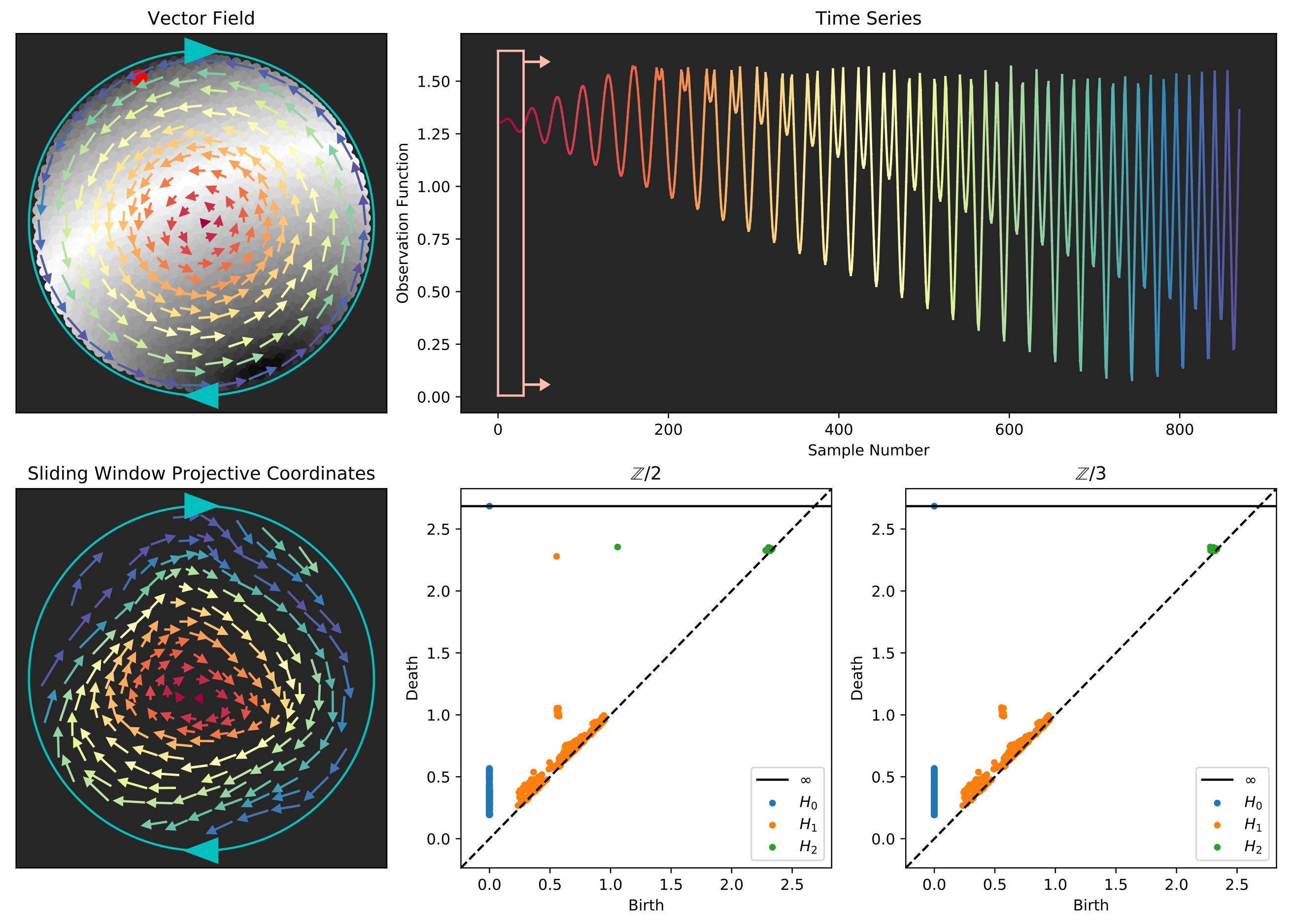}
        \caption{An observation function on $\RP^2$ which is the geodesic distance from a point $\hat{x}$ drawn in red.}
        \label{fig:ProjDist}
    \end{figure}

    We can extend the scheme that we used in Example~\ref{ex:sphere} to the projective plane $\mathbb{RP}^2$ by taking a flow only on the upper hemisphere and performing the antipodal identification at the equator $x \sim -x$.  The flow $\psi_t$ is the same, but the observation function changes to
    
    \begin{equation}
        G(\theta, \phi) = \cos^{-1} \left| \cos(\phi)\sin(\theta)\cos(\hat{\phi})\sin(\hat{\theta}) + \sin(\phi)\sin(\theta)\sin(\hat{\phi})\sin(\hat{\theta}) + \cos(\theta)\cos(\hat{\theta}) \right|
    \end{equation}
    
    Figure~\ref{fig:ProjDist} shows this result, in which a single highly persistent point is present for both $H_1$ and $H_2$ using $\mathbb{Z}/2$ coefficients, but in which none are present for $\mathbb{Z}/3$, which is a correct signature of $\RP^2$.  Interestingly, the quotient identification is visible in the time series itself; the time series in Figure~\ref{fig:ProjDist} can be obtained from the time series in Figure~\ref{fig:SphereDist} by reflecting values above the line $y = \pi/2$ across that line.  This is because the maximum distance between any two points on $\RP^2$ is $\pi/2$.  Additionally, both the sphere time series and the M{\"o}bius loop time series ($\cos(t) + a \cos(2t)$) are visible in Figure~\ref{fig:ProjDist}.  The time series starts off in a spiral, which fills out a disc, and this disc transitions to a spiraling M{\"o}bius loop time series which fills out the strip.  This visually reflects the fact that $\RP^2$ is the connected sum of a disc and the boundary of a cross-cap.  We will use  a similar intuition to explain the Klein bottle time series in Section~\ref{sec:kleinbottle}.
    \end{example}

    \begin{example}
    Genus 2 surface $\TT \# \TT$

    \begin{figure}[!htb]
        \centering
        \includegraphics[width=\textwidth]{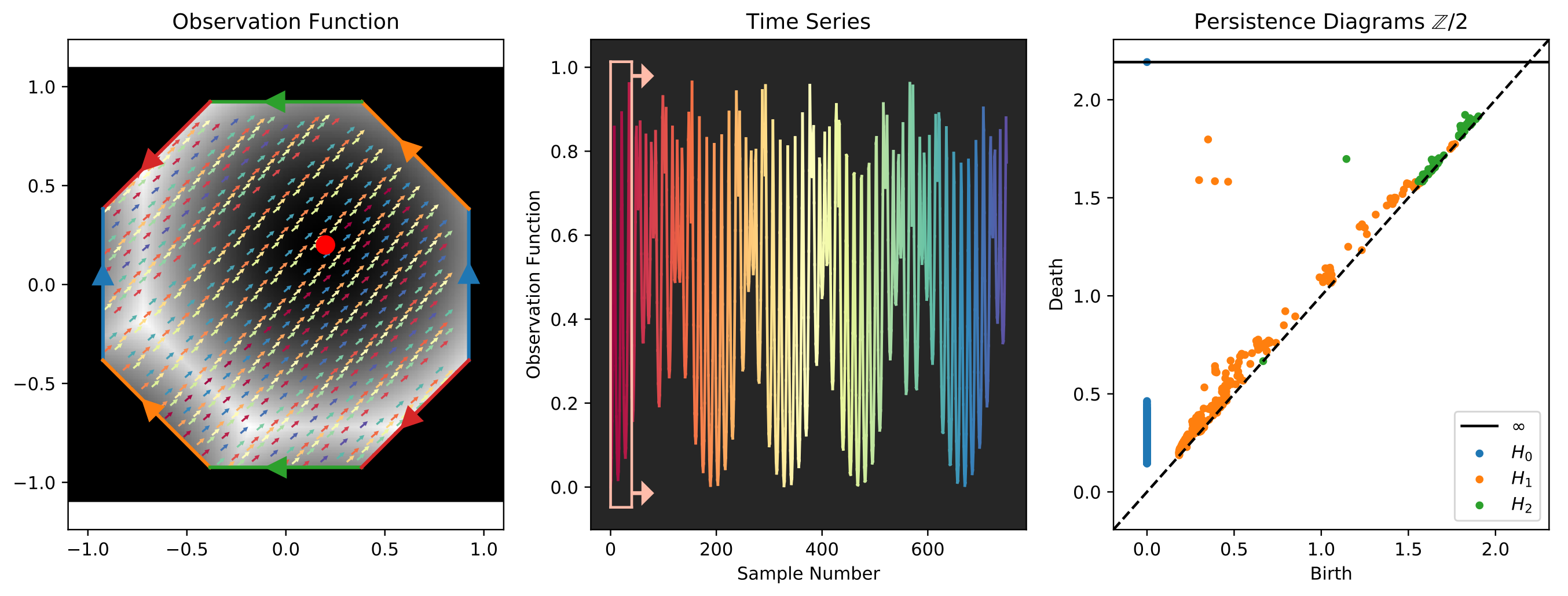}
        \caption{An example of a time series resulting from a dense flow on the 2-holed torus, using the flat squared Euclidean distance \cite{zorich2006flat} from an observation point (shown as a red dot).}
        \label{fig:TwoHoledTorusDist}
    \end{figure}

    Finally, we show a time series whose sliding window embedding lies on the two holed torus.  We use an irrational flow with slope $(dt, dt \sqrt{3}/2)$, with an observation function as the squared flat metric on the fundamental domain \cite{zorich2006flat} represented by an octagon with opposite sides identified.  Figure~\ref{fig:TwoHoledTorusDist} shows the result, in which four highly persistent dots are visible in $H_1$ and a single persistent dot is visible in $H_2$, matching what is expected of the homology of a genus 2 surface.

    \end{example}

    \section{Main Theorem: Characterizing Good Observation Functions}
    \label{sec:goodobservations}

    As our main theoretical contribution, we now state more general conditions for good observation functions.  Let $M$ be a compact manifold of dimension $m$, $G:M \rightarrow \RR$ a smooth function, and $X$ a vector field with flow $\psi_t$. Applying $G$ to an integral curve $\gamma_p(t) = \psi_t(p)$ through a point $p$
    yields a real-valued function
    \[g_p := G\circ\gamma_p\]
    in $t$, the \textit{observation curve} of $p$. For sufficiently nice $G$ and $X$, one can recover the point $p$ from a finite uniform sampling of $g_p$. More precisely, the \textit{Takens map} $\Psi_\tau^N: M \rightarrow \RR^{N+1}$ defined by
    \[\Psi_\tau^N(p) = \big(g_p(0), g_p(\tau), g_p(2\tau), \ldots, g_p(N\tau)\big)\]
    is an embedding for some dimension $N>0$ and flow time $\tau>0$. For such $G$ and $X$ we say $G$ is a \textit{good observation} for $X$.
    
    \subsection{Motivation for the approach}
    As a simple example, take $M=S^1 = \RR/2\pi\ZZ$, $\psi_t(x) = x+t$, and $G(x) = \cos(x)$. The point $x$ is uniquely determined by sampling the two values $g_x(0)=\cos(x)$ and $g_x(\pi/2)= -\sin(x)$ and the Takens map
    \[\Psi_{\pi/2}^1(x) = (\cos(x), -\sin(x)) \]
    is an embedding, so $G$ is a good observation.

    On the other hand, the doubly periodic function $G(x) = \cos(2x)$ is not a good observation function. Indeed, any integral curve $g_x(t)$ is invariant under a $\pi$-shift of $x$, as $G$ cannot distinguish between any flow of $x$ and $x+\pi$. In fact, the good observation functions on $S^1$ for the rotational dynamic are precisely ones with minimum period $2\pi$

    In higher dimensions the task of recovering $p$ from $g_p$ becomes less clear. Consider the torus $\TT = S^1\times S^1$ and $G:\TT \rightarrow \RR$ given by
    \[G(x,y) = \cos(x) + \cos(y)\]
    and $\psi_t$ an irrational flow
    \[\psi_t(x,y) = (x+\alpha t, y+ \beta t)\]
    and thus for $p = (x,y) \in \TT$ we have the observation curve
    \[g_p(t) = \cos(x+\alpha t) + \cos(y + \beta t) \]
    For $G$ to be good, there must be a $\tau$ such that each $p$ is uniquely determined by sampling $G$ along the integral curve $\gamma_p$ at finitely many $\tau$-steps. Since we are free to shrink $\tau$ and increase $N$, it is natural to examine infinitesimal changes of $G$ along the flow $\psi_t$. The derivatives
    \begin{align*}
    g_p(0) &= \cos(x) + \cos(y)\\
    g_p'(0) &= -\alpha\sin(x) - \beta\sin(y)\\
    g_p^{(2)}(0) &= -\alpha^2\cos(x) - \beta^2\cos(y)\\
    g_p^{(3)}(0) &= \alpha^3\sin(x) + \beta^3\sin(y)
    \end{align*}
    up to 3\ts{rd} order yield the linear equation
    \[\begin{pmatrix}1 & 1 & 0 & 0 \\ -\alpha^2 & -\beta^2 & 0 & 0 \\ 0 & 0 & -\alpha & -\beta \\ 0 & 0 & \alpha^3 & \beta^3\end{pmatrix}
    \begin{pmatrix} \cos(x) \\ \cos(y) \\ \sin(x) \\ \sin(y)\end{pmatrix} = \begin{pmatrix} g_p(0) \\ g_p'(0) \\ g_p^{(2)}(0) \\ g_p^{(3)}(0) \end{pmatrix} \]
    Equivalently, over $\CC$, the linear system
    \[\begin{pmatrix}1 & 1 & 1 & 1 \\ i\alpha & -i\alpha & i\beta & -i\beta \\ -\alpha^2 & -\alpha^2 & -\beta^2 & -\beta^2 \\ -i\alpha^3 & i\alpha^3 & -i\beta^3 & i\beta^3 \end{pmatrix}
    \begin{pmatrix} e^{ix} \\ e^{-ix} \\ e^{iy} \\ e^{-iy}\end{pmatrix} = \begin{pmatrix} g_p(0) \\ g_p'(0) \\ g_p^{(2)}(0) \\ g_p^{(3)}(0) \end{pmatrix} \]
    has invertible Vandermonde matrix and one can solve for $e^{ix}$ and $e^{iy}$. Therefore $(x,y)$ is uniquely determined by $g_p^{(k)}(0)$'s. Choosing $\tau$ small enough so that $g_p(\tau)$ is close to the 3\ts{rd} order Taylor polynomial of $g_p$ about $0$, we see that $p$ is uniquely determined (modulo $2\pi$) by a $\tau$-uniform finite sampling of $g_p$.

      \subsection{Main theorem and proof}
    The above calculation illustrates our approach to determining whether $G$ is good: by studying the Taylor coefficients $g_p^{(k)}$. Note that $g_p^{(k)}(t)$ is the $k$-fold derivation of $X$ applied to $G$ at $\psi_t(p)$, i.e., in Lie derivative notation,
    \[g_p^{(k)}(t) = \Ld_X^kG(\psi_t(p)).\]
    $\Ld_X$ is the linear operator on tensor fields which measures infinitesimal change along $X$, i.e. if $T$ is a tensor then

    \[\Ld_XT(p) = \frac{d}{dt}\bigg|_{t=0}((\psi_{-t})_*T_{\psi_t(p)})\]
    Writing $dG$ for the differential of $G$, and $\wedge$ for exterior product, we now state the main result:

    \begin{theorem}\label{thm:main}
    The Takens map $\Psi_{\tau}^{N}$ is an embedding for some $N>0$ and flow time $\tau>0$, if the following conditions hold:
        \begin{enumerate}
        \item For any point of $p\in M$ there is an $m$-tuple $J\in \ZZ_{\geq 0}^m$ of nonnegative integers such that the $m$-form
        \[\Ld_{X}^{\wedge J}dG := \bigwedge_{j\in J}\Ld_X^jdG\]
        is nonzero at some point on the integral curve $\gamma_p(s)$.

        \item For any pair of distinct points $p, q\in M$ the observation curves $g_p(s)$ and $g_q(s)$ are not identical.
        \end{enumerate}
    \end{theorem}

    \begin{proof} For $\Psi_\tau^N$ to be an immersion, the cotangent vectors
    \[dG|_p, \hspace{1mm} d(G\circ\psi_\tau)|_p, \hspace{1mm}d(G\circ\psi_{2\tau})|_p, \hspace{1mm}\ldots, \hspace{1mm}d(G\circ\psi_{N\tau})|_p\in T^*_pM\]
    must span an $m$-dimensional space for all $p \in M$. Equivalently, for any point $p\in M$ there must be a strictly increasing $m$-tuple $I=(i_1, i_2, \ldots, i_m)\in \ZZ_{\geq 0}^m$ of indices such that the determinant $m$-form $\bigwedge d(G\circ\psi_{i_k\tau})$ does not vanish at $p$, i.e.
    \[\omega_{p}^I(\tau) := \bigwedge_{i_k\in I} d(G\circ\psi_{i_k\tau})|_{p} \neq 0.\]
    We require that $I$ be strictly increasing because a wedge product containing identical factors is zero.

    The idea is to perform a convolution of the Taylor series of the cotangent curves $d(G\circ\psi_{i_kt})$ and to use condition 1 above to choose sufficiently small $\tau$ so that $\omega_p^I(\tau)\neq 0$ for some $I$. By compactness of $M$, one makes a uniform choice of small $\tau$ so that $\Psi_\tau$ is immersive and each observation curve is distinguished on some integer multiple of $\tau$, thereby making $\Psi_\tau^N$ injective.

    Let $s\geq 0$ be a time parameter for $p$ such that
    \[\Ld_{X}^{\wedge J}dG\]
    is nonzero at $\gamma_p(s)$. Write $\tilde{p} = \gamma_p(s)$ and $\mathcal{J}_n$ for the set of all strictly increasing $m$-tuples $J = (j_1, j_2, \ldots, j_m)$ with degree
    \[j_1 + j_2 + \ldots + j_m = n\] satisfying \[\Ld_{X}^{\wedge J}dG|_{\tilde{p}} \neq 0\]
    Fix $n> 0$ to be the minimal integer for which $\mathcal{J}_n$ is nonempty (possible by condition 1 above).

    Let $A(t)$ be the $m$ by $(n+1)$ matrix with $(k, j)$\ts{th} entry
    \[A_{k,j}(t) = \frac{i_k^{j-1}(t-s)^{j-1}}{(j-1)!}\]
    and $L:T_{\tilde{p}}M \rightarrow \RR^{n+1}$ the linear map given by $\Ld_X^{j-1}dG|_{\tilde{p}}$ in the $j$\ts{th} coordinate,
    \[
    L = \left(dG|_{\tilde{p}}, \Ld_X^1 dG|_{\tilde{p}}, ..., \Ld_X^n dG|_{\tilde{p}}\right).
    \]
    
    So the $k$\ts{th} component of the composition $A(t)\circ L$, viewed as an $m$-tuple of $t$-dependent cotangent vectors, yields the $n$\ts{th} order Taylor polynomial about $t=s$ of the cotangent curve $d(G\circ\psi_{i_kt})|_{p}$:
    \[
    A(t) \circ L = \sum_{j=0}^n \frac{i_k^j (t-s)^j}{j!}\Ld_X^j dG|_{\tilde{p}}.
    \]

    By Cauchy-Binet formula applied to $A(t)$ and $L$, the top exterior product
    \[\omega_{p}^I(t)= \bigwedge_{i_k\in I} d(G\circ\psi_{i_kt})|_{p}
    \]
    has $n$\ts{th} order Taylor series expansion about $t=s$ with $n$\ts{th} coefficient
    \[C_n = \frac{\det(V)}{a_n} \sum_{J\in \mathcal{J}_n}|I^{J}|\cdot \Ld_{X}^{\wedge J}dG|_{\tilde{p}}\]
    where
    \begin{itemize}
    \item$a_n$ is a nonzero constant depending only on $n$
    \item $|I^J| = \prod i_k^{j_k-k+1}$
    \end{itemize}
    and \[\det(V) = \prod_{k<k'} (i_{k'}-i_k) \neq 0\]
    is the nonzero determinant of the $m\times m$ Vandermonde matrix $V$ with $(k,j)$\ts{th} entry
    \[V_{k, j}={i_k}^{j-1}\]
    where we take $0^0 = 1$.

     By the minimality assumption on $\mathcal{J}_n$, all the lower degree Taylor coefficients, which contain $\Ld_X^K dG|_{\tilde{p}} = 0$ for $m$-tuples $K$ with degree strictly less than $n$,
    \[C_j = 0 \text{ for $j<n$}\]
    are zero. So the Taylor expansion of $\omega_{p}^I(t)$ has the form
    \[\omega_{p}^I(t) = (t-s)^nC_n + R_p^n(t)\]
    where $R_p^n(t)$ is the $n$\ts{th} order Taylor error term with vanishing limit
    \[\lim_{t\to s}\frac{R_p^n(t)}{(t-s)^n} = 0\]

    For for suitable choice of $I$, there will be a dominating term in the sum over $\mathcal{J}_n$ such that $C_n$ is nonzero. For $\tilde{J}\in\mathcal{J}_n$ the colexigraphically maximal element of $\mathcal{J}_n$, let $a$ be the maximal index such that $j_a < \tilde{j}_a$, for all $J < \tilde{J}$. Choose $I$ by making all terms right of $a-1$ large, so that
    \[|I^{J}| << |I^{\tilde{J}}|\]
    for all $J < \tilde{J}$
    \[\sum_{J\in \mathcal{J}_n}|I^{J}|\cdot \Ld_{X}^{\wedge J}dG|_{\tilde{p}} \neq 0.\]
    So the $n$\ts{th} Taylor coefficient
    \[C_n \neq 0\]
    is nonzero.
     Hence we may choose a time $\eta>s$ sufficiently close to $s$ so that the Taylor error $R_p^n(\eta)$ is small and the inequality
    \[\omega_{q}^I(\eta) \neq 0\]
    holds for all $q$ in a neighborhood of $p$, and this property remains invariant under shrinking $\eta$ closer to $s$. By compactness of $M$ there is a finite collection of triples $(I_r, \eta_r, s_r)$ such that the collection of $m$-forms
    \[\{\omega^{I_r}(\eta_r)\}\]
    do not all vanish at any given point of $M$ and the cotangent vectors
    \[\{d(G\circ\psi_{i_k\eta_r})|_q\}_{i_k \in I_r}\]
    specified by $I_r$ are linearly independent. Choose $\tau>0$ small enough so that there is an integer multiple of $\tau$ lying in the interval $(s_r, \eta_r)$ for each $r$. Then the Takens map $\Psi_\tau^N$ is an immersion for all $N>0$ bounding $I_r$ and $\eta_r/\tau$.

    So $\Psi_\tau^N$ is locally injective and the difference map
    \[\Psi_\tau^N(p) - \Psi_\tau^N(q)\]
    does not vanish for all $p\neq q$ in an open neighborhood $U$ of the diagonal in $M\times M$, and this property is invariant under scaling $N\mapsto Nd$ and $\tau\mapsto \tau/d$ for an integer $d>0$ (with $U$ fixed).

    For distinct $(p, q)\in M\times M\setminus U$, we may shrink $\tau$ so that $g_p$ and $g_q$ are distinguished on some integer multiple of $\tau$ and $\Psi_\tau(p) \neq \Psi_\tau(q)$. By compactness of $M\times M\setminus U$, there is a uniform choice of $\tau$ and $N$ making $\Psi_\tau^N$ injective, hence an embedding.

    \end{proof}

    \begin{remark}
    While one can provide a lower bound for the dimension $N$ needed to yield a Takens embedding, the formula depends in a complicated way on $G$ and $X$. In practice, choosing sufficiently large $N$ and small $\tau$ amounts to a dense sampling of a discrete time series.
    \end{remark}

    \section{An Application to Surfaces via Fourier theory}
    \label{sec:fourier}

    Now that we have our theory in hand, we can examine another class of observation functions which are constructed from Fourier modes, in addition to our distance-based observation functions in Section~\ref{sec:distancesobs}.

    \subsection{The Torus}



    We start by characterizing all smooth observations $G:\TT \rightarrow \RR$ for a vector field $X$ of irrational flow
    \[\psi_t(x,y) = (x+\alpha t, y+\beta t)\]
    yielding toroidal delay embedding. For $G$ write the Fourier expansion
    \[G(x,y) = \sum_{(n,m)\in \ZZ^2} \hat{G}(n,m)\cdot \exp(i(nx+my))\]
    where $\hat{G}(n,m)\in\CC$ is the $(n,m)$\ts{th} Fourier coefficient of $G$. Set
    \[\Supp\hat{G} = \{ (n,m) \in \ZZ^2 \mid \hat{G}(n,m) \neq 0 \}\]
    the \textit{support} of $\hat{G}$.
    \begin{theorem}\label{thm:fourier}
    A smooth function $G:\TT \rightarrow \RR$ is a good observation for an irrational winding if and only if the support $\Supp\hat{G}$ of the Fourier coefficients generates $\ZZ^2$ as an abelian group.
    \end{theorem}

    \begin{proof}\label{pf:fourier}
    Write $e_{n,m} = \exp(i(nx+my))$ for the $(n,m)\ts{th}$ Fourier basis element. The $k$-fold Lie derivative $\mathcal{L}_X^kG$ has Fourier coefficient
    \[\widehat{\mathcal{L}_X^kG}(n,m) = i^k(n\alpha + m\beta)^k\cdot \hat{G}(n,m)\]
    and thus Fourier expansion
    \[\mathcal{L}_X^kG = \sum_{(n,m) \in \ZZ^2} i^k(n\alpha + m\beta)^k\hat{G}(n,m)\cdot e_{n,m}\]
    Since $\alpha/\beta$ is irrational, the coefficients
    \[c_{n,m} = i\cdot(n\alpha + m\beta)\]
    are nonvanishing and pairwise distinct. Therefore the Vandermonde matrix with $(n,m)\times j\ts{th}$ entry
    \[(c_{n,m}^k)\]
    is nonsingular and the projection
    \[G*e_{n,m} = \hat{G}(n,m)\cdot \exp(i(nx+my))\]
    can be written as an infinite sum
    \begin{equation}\label{eq:inf-lin-combo}\hat{G}(n,m)\cdot \exp( i(nx+my) = \sum_{j=0}^\infty b_j\Ld_X^jG\end{equation}
    Hence the values of $\Ld_X^kG$ on a point $(u,v)\in \TT$ uniquely determine
    \[\hat{G}(n,m)\cdot e^{i(nu+mv)}\]
    If $\Supp\hat{G}$ generates $\ZZ^2$, then there is some finite product
    \[\prod_{(n_j,m_j)\in \Supp\hat{G}}e^{i(n_ju+m_jv)} = e^{iu}\]
    and thus $u$, and similarly $v$, are uniquely determined modulo $2\pi$ by the observation curve $G\circ\gamma_{u,v}$ and condition 2 of Theorem \ref{thm:main} above is satisfied.

    If $d\Ld_X^jG\wedge d\Ld_X^kG$ vanishes at $p$ for all $j,k\geq 0$, then by equation \ref{eq:inf-lin-combo} above, the $2$-form
    \[d(G*e_{n,m})\wedge d(G*e_{n',m'}) = \det\begin{pmatrix}n & m \\n' & m'\end{pmatrix}\hat{G}(n,m)\hat{G}(n',m')\cdot e_{n,m} e_{n',m'}\]
    also vanishes at $p$ for all pairs $(n,m), (n',m')\in \ZZ^2$. Thus
    \[\det\begin{pmatrix}n & m \\n' & m'\end{pmatrix} = 0 \textrm{  for all $(n,m), (n',m') \in \Supp \hat{G}$}\]
    and $\Supp\hat{G}$ cannot generate $\ZZ^2$. So condition 1 of Theorem \ref{thm:main} is satisfied if $\Supp\hat{G}$ generates $\ZZ^2$.

    Conversely, suppose $\Supp\hat{G}$ does not generate $\ZZ^2$. By the classification of finitely generated abelian groups, there is a $\ZZ$-basis
    \[(n_1,m_1), (n_2, m_2)\]
    for $\ZZ^2$ such that $\Supp\hat{G}$ is generated by
    \[a\cdot (n_1, m_1), b\cdot(n_2,m_2)\]
    where $a$ and $b$ are integers not both $\pm 1$. Then there is some $(u,v) \notin 2\pi\ZZ^2$ such that
    \[\begin{pmatrix}an_1 & am_1 \\bn_2 & bm_2\end{pmatrix}\cdot \begin{pmatrix}u \\ v\end{pmatrix}\]
    takes values in $2\pi\ZZ$, so that $\exp(i(nu+mv)) = 1$ for all $(n,m) \in \Supp\hat{G}$. So for any point $(x,y)\in \TT$, $(x+u,y+v)\in \TT$ is a distinct point with the same observation curve, and no Takens map can distinguish between $(x,y)$ and $(x+u, y+v)$.
    \end{proof}
    
    \begin{remark}
        Theorem \ref{thm:fourier} can be strengthened to include rational windings. In this case one cannot expect the delay mapping to recover all Fourier modes of an observation function, but only those which are coprime to the slope of the winding.
    \end{remark}

    \begin{remark}
        For the irrational winding on the torus, the Koopman eigenfunctions are given by the Fourier basis. The Vandermonde inversion in equation (\ref{eq:inf-lin-combo}) above shows that the Fourier modes of an observation are determined by its delay mapping.  We are not aware of such a connection between Takens and Koopman, though it seems natural in this context.
    \end{remark}

    By Theorem \ref{thm:fourier}, whether or not $G$ is good for an irrational flow depends only on the support $\Supp\hat{G}$. The quasiperiodic function
    \begin{equation}
    \label{eq:torusirrationalexample}
    g(t) = \cos \sqrt{2} t  + \cos t
    \end{equation}
    is the observation of $G(x,y) = \cos(x)+\cos(y)$ along the irrational flow $(\sqrt2t, t)$ on the planar torus $\TT=\RR^2/2\pi\ZZ^2$. A point cloud densely sampled from the sliding window $\SW_1^{10}g(t)$ coordinates given by 10 uniform shifts of $g(t)$ yields a curve in $\RR^{10}$ with toroidal persistence.
    \begin{figure}[!htb]
        \centering
        \includegraphics[width=\textwidth]{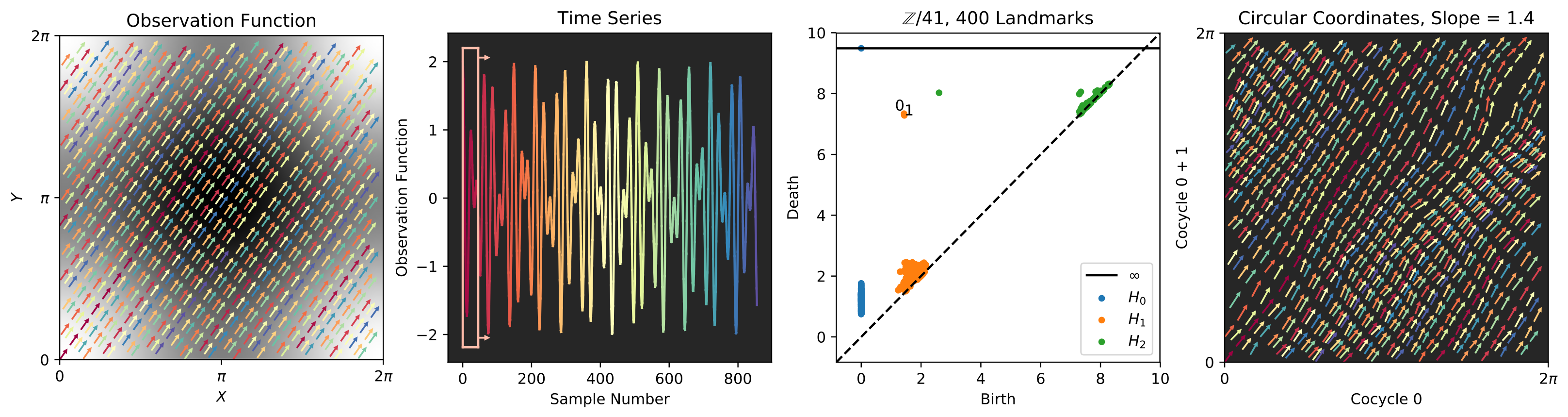}
        \caption{The observation function $\cos(x) + \cos(y)$ for the same flow as Figure~\ref{fig:TorusDist}}
    \end{figure}

    \subsection{The Klein bottle}
    \label{sec:kleinbottle}

    As in our example in Section~\ref{sec:projcoords}, we write the Klein bottle $\KK$ as the quotient of the torus $\TT$ by the automorphism $\kappa: (x,y) \mapsto (x+\pi, -y)$. The irrational flow on the $\TT$ is not $\kappa$-invariant since $\kappa$ is orientation reversing in the $y$ coordinate. To approximate the shallow flow in Figures \ref{fig:KleinDist} and \ref{fig:KleinDistFail} above, we construct a vector field which flows cyclically along a repellor $y=0$ and an attractor $y=\pi$ by restricting a linear flow to the fundamental domain $[0,2\pi]\times[0, \pi]$ and flatten it out on the boundary circles $y=0, \pi$. For $\alpha,\beta\in \RR$ with $0 < \alpha/\beta<<1$ irrational, let $X_\epsilon$ be a vector field on the rectangle given by
    \[ X_\epsilon(x,y)= \begin{cases}
        (\alpha, \rho(y)) & 0 \leq y \leq \epsilon\\
        (\alpha, \beta) & \epsilon < y \leq \pi-\epsilon \\
        (\alpha, \rho(\pi+\epsilon-y)) & \pi-\epsilon < y \leq \pi
    \end{cases}
    \]
    where $\rho$ is a smooth function on a neighborhood of $[0, \epsilon]$ with $\rho(0) = 0$, $\rho(\epsilon) = \beta$ making $X_\epsilon$ smooth. For example, $\rho = \beta\exp(1/(y/\epsilon-1)^2-1)$. Then $X_\epsilon$ extends uniquely to a $\kappa$-invariant vector field on on $\TT$, and therefore induces a vector field on $\KK$.


    \begin{theorem}
    \label{thm:klein}
    Let $G:\TT \rightarrow \RR$ be a $\kappa$-invariant function on $\TT$. For fixed $N\tau$, the Takens map
    \[\Psi_\tau^N: \KK \rightarrow \RR\]
    induced by $G$ and $X_\epsilon$ for arbitrarily small $\epsilon$ and slope $\alpha/\beta<<1$ is an embedding if and only if the following conditions hold:
    \begin{enumerate}
        \item $G(x,\pi)$ and $G(x,0)$ have period $\pi$ in $x$ and do not differ by a shift
        \item $\Supp \hat{G}$ generates $\ZZ^2$
    \end{enumerate}
    \end{theorem}

    \begin{proof}
    Suppose $G$ is good for $X_\epsilon$. Since $X_\epsilon$ flows horizontally at $y=0,\pi$, condition 1) must hold so that each point is uniquely determined by its observation curve. Condition 2) must hold as well, since $X_\epsilon$ is given by an irrational winding away from
    the $\epsilon$-neighborhood of $y=0,\pi$ and the same argument as in Theorem \ref{thm:fourier} above applies for sufficiently shallow slope $\alpha/\beta$ because $N\tau$ is fixed.

    Conversely, suppose conditions 1) and 2) hold. $X_\epsilon$ is given by an irrational flow away from the $\epsilon$-neighborhood of $y=0,\pi$. Furthermore, any point in the $\epsilon$-strip with $y\neq 0,\pi$ may be flowed to a point where $X_\epsilon$ has irrational slope. The same argument as in Theorem \ref{thm:fourier} shows that the Takens map restricts to an embedding on $y\neq 0, \pi$.

    By condition 1, the observation curve of a point $(x,y)$ where $y=0,\pi$ uniquely determines $x$ modulo $\pi$, and is periodic and therefore distinct from any observation curve for $y\neq 0,\pi$. So each point is uniquely determined by its observation curve as per condition 2) of Theorem \ref{thm:main}.

    It remains to show that the Takens map is immersive at $y=0,\pi$. If not, then $\frac{\partial G}{\partial y}$ vanishes on the circles $y=0,\pi$, a neighorhood about which $\Psi_\tau^N$ would fail to immerse, a contradiction.
    \end{proof}

    According to Theorem \ref{thm:klein}, the ``simplest'' $\kappa$-symmetric good observation is
    \begin{equation}
    \label{eq:kleinfourier}
    G(x,y) = \cos 2x + \cos x\sin y + \cos y.
    \end{equation}
    Indeed, the Fourier coefficients of $G$ are supported at $(\pm2,0), (\pm1, \pm1), (0, \pm1)$, which generates $\ZZ^2$. Along the limit cycles we have $G(x, 0) = \cos 2x + 1$ and $G(x,0) = \cos 2x - 1$, which are distinct and doubly periodic.

    Intuitively, the $\cos 2x$ term is responsible for delay-mapping the limit cycles $y=0,\pi$ via a double covering.  Without this term, the boundary $G(x, 0) = 1$, $G(x, \pi) = -1$ along the bottom and top boundaries, respectively.  Not only are these boundaries no longer identified, but they also each map to a single point, turning the Klein bottle into a sphere.  The delay mapping of $\cos x\sin y$ fills two M\"{o}bius strips in conjuction with $\cos(2x)$, while the $\cos(y)$ term serves to ``separate'' the M\"{o}bius strips, as shown in the right hand side of Figure \ref{fig:eyesoftheworld}.

    We can also see this by parameterizing the flow by a single variable $t = x$ and examining the time series directly.  In this case, the time series is
    \begin{equation}
    g(t) = \cos(2t) + \cos(t) \sin \left( \frac{\alpha}{\beta} t \right) + \sin \left( \frac{\alpha}{\beta} t \right)
    \end{equation}
    for $\epsilon < \frac{\alpha}{\beta} t < \pi - \epsilon$.  Over small ranges of $t$, the sine terms are approximately constant.  The time series is then of the form $\cos(2t) + a \cos(t), |a| < 1$; that is, its sliding window embedding locally parameterizes the boundary of a M\"{o}bius strip \cite{perea2015sliding}.  As it moves further along, $a$ changes, and so it fills out the strip.

    \begin{figure}[!htb]
        \centering
        \includegraphics[width=\textwidth]{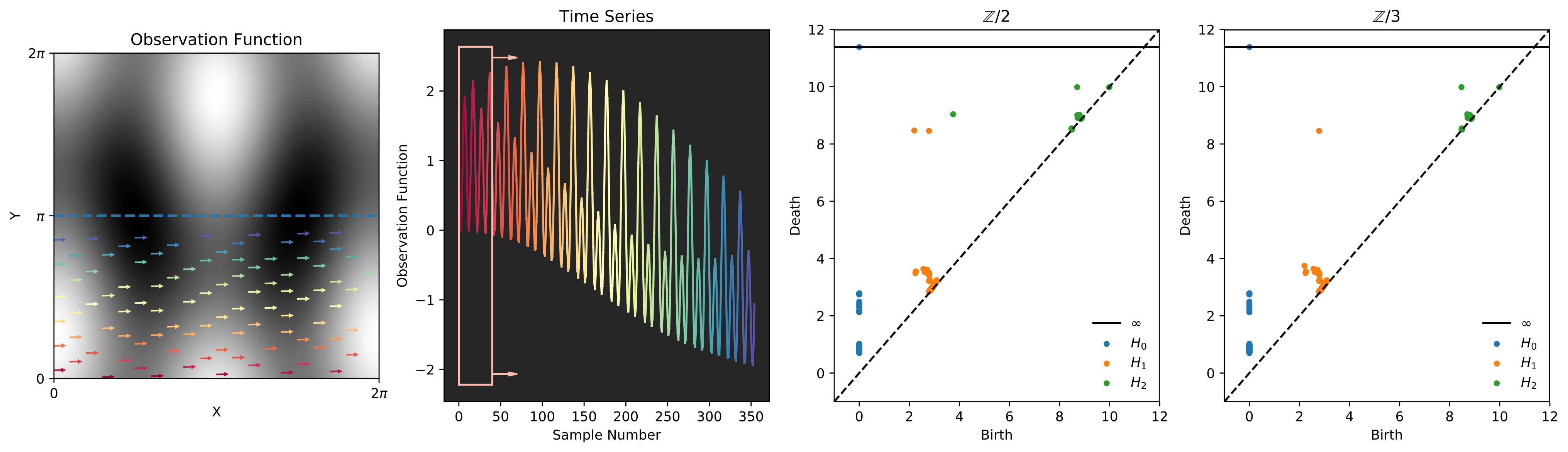}
        \caption{The observation function $G(x,y)$ from Equation~\ref{eq:kleinfourier} for the same flow as Figure~\ref{fig:KleinDist}}. Indeed the good observation function reproduces $\mathbb{K}$, as evidenced by the persistence diagram.
    \end{figure}

    \begin{figure}[!htb]
        \label{fig:eyesoftheworld}
        \centering
        \includegraphics[width=\textwidth]{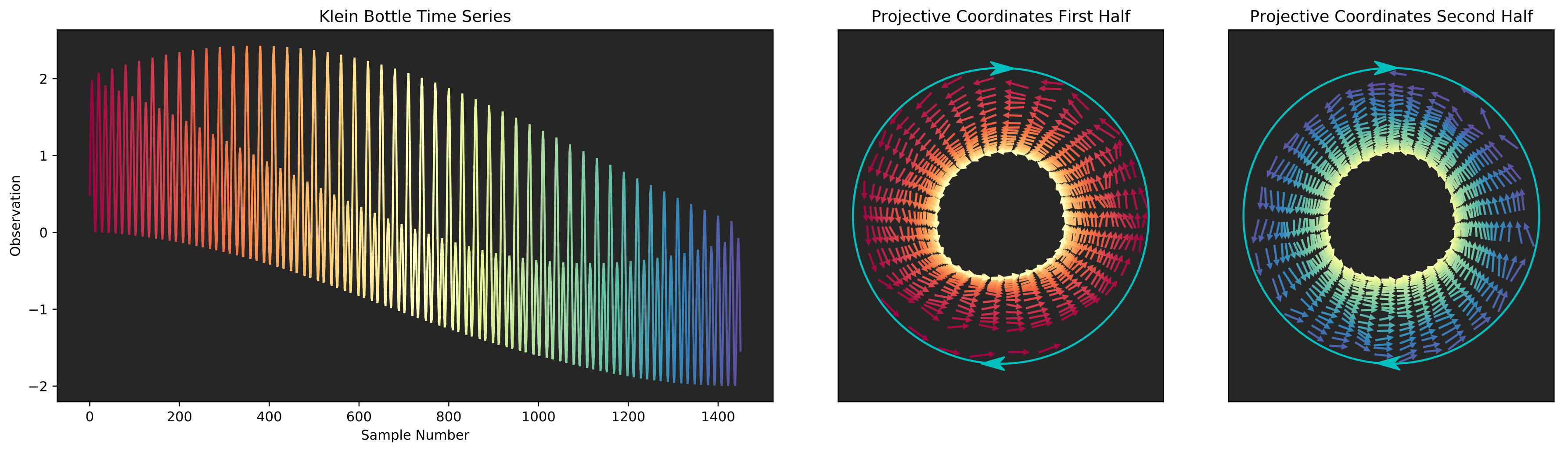}
        \caption{Projective coordinates for a Klein bottle with observation function specified in Equation~\ref{eq:kleinfourier}.  We plot the first half in the second subplot, which traces a M\"{o}bius strip from its core to its boundary, shown in yellow.  Then, that boundary is glued to a second M\"{o}bius strip, which corresponds to the second half of the time series, as shown on the right.}
    \end{figure}

    \section{Discussion}

    It is clear what circular and toroidal observations look like in the time domain, and as we have mentioned, there are many applications that take advantage of this knowledge.  The theory developed in this paper has enabled us to move beyond this and to develop examples of signals recovering other manifolds.
    
    Also, by showing the existence of time series whose ``attractors'' are on twisted spaces, we also provide further motivation for TDA time series users to move beyond exclusively using $\mathbb{Z} / 2 \mathbb{Z}$ in TDA.  The latter is the default option across most applications of TDA in time series analysis, but it is possible that these pipelines are blind to important features, as some of our examples show.  
    
    Moreover, just as circular and toroidal sliding window embeddings have interpretations in terms of physical phenomena, the presence of Klein bottles, Moebius strips, spheres, projective planes, etc, should also have practical meaning. It is unlikely that one could recognize the significance of these time series in the wild without such examples in hand, and being primed as such makes it more likely that we will be able to discover physical examples where non-orientable state spaces are natural.
    
    Finally, we note that not only do we have a method for producing time series recovering other manifolds, which we have validated empirically using persistent homology and Eilenberg MacClane coordinates, but the method is backed by a theorem that indicates exactly when it will succeed/fail.


    \bibliographystyle{plain}
    \bibliography{takens-article}

    \end{document}